\documentclass[11pt, a4paper]{article}
\usepackage{amsmath}
\usepackage{amsthm}
\usepackage{amssymb}
\usepackage{array}
\usepackage{multirow}
\usepackage{graphicx}
\usepackage{epsfig}
\usepackage{mathrsfs}
\usepackage{hhline}
\usepackage{longtable}
\usepackage{array}
\usepackage{cite}
\usepackage{enumerate}
\usepackage{xcolor}
\usepackage{float}

\usepackage{graphicx}
\graphicspath{%
    {converted_graphics/}
    {/}
}

\newcolumntype{"}{@{\hskip\tabcolsep\vrule width 1pt\hskip\tabcolsep}}

\usepackage[mathlines]{lineno}
\modulolinenumbers[2]
\newcommand*\patchAmsMathEnvironmentForLineno[1]{%
\expandafter\let\csname old#1\expandafter\endcsname\csname #1\endcsname  \expandafter\let\csname oldend#1\expandafter\endcsname\csname end#1\endcsname  \renewenvironment{#1}%
{\linenomath\csname old#1\endcsname}%
{\csname oldend#1\endcsname\endlinenomath}}%
\newcommand*\patchBothAmsMathEnvironmentsForLineno[1]{%
\patchAmsMathEnvironmentForLineno{#1}%
\patchAmsMathEnvironmentForLineno{#1*}}%
\AtBeginDocument{%
\patchBothAmsMathEnvironmentsForLineno{equation}%
\patchBothAmsMathEnvironmentsForLineno{align}%
\patchBothAmsMathEnvironmentsForLineno{flalign}%
\patchBothAmsMathEnvironmentsForLineno{alignat}%
\patchBothAmsMathEnvironmentsForLineno{gather}%
\patchBothAmsMathEnvironmentsForLineno{multline}%
}

\def\Z{\mathbb{Z}}

\definecolor{vividviolet}{rgb}{0.62, 0.0, 1.0}
\def\nt{\noindent}
\def\ms{\medskip}
\def\rsq{\hspace*{\fill}$\blacksquare$\medskip}

\newtheoremstyle{de}
  {10pt}          
  {10pt}  
  {\rm}  
  {}
  {\bf}  
  {. }    
  { }    
  {}     
\theoremstyle{de}

\newtheorem{example}{Example}[section]

\newtheorem{problem}{Problem}[section]

\newtheoremstyle{theorem}
  {10pt}          
  {10pt}  
  {\it}  
  {}
  {\bf}  
  {. }    
  { }    
  {}     
\theoremstyle{theorem}

\numberwithin{equation}{section}
\def\Z{\mathbb{Z}}

\newtheorem{theorem}{Theorem}[section]
\newtheorem{lemma}[theorem]{Lemma}

\newtheorem{corollary}[theorem]{Corollary}
\newtheorem{conjecture}{Conjecture}[section]
\numberwithin{equation}{section}

\def\ms{\medskip}

\def\nt{\noindent}

\def\a{\alpha}
\def\b{\beta}
\def\n{\nu}

\def\m{\mu}

\def\k{\kappa}

\def\g{\gamma}

\def\sg{\sigma}
\def\r{\rho}
\def\to{\rightarrow}

\newcommand{\K}{K\lower0.2cm\hbox{4}\ }

\newcommand{\ben}{\begin{enumerate}}

\newcommand{\een}{\end{enumerate}}
\newcommand{\bit}{\begin{itemize}}
\newcommand{\eit}{\end{itemize}}
\newcommand{\bea}{\begin{eqnarray*}}
\newcommand{\eea}{\end{eqnarray*}}
\newcommand{\bear}{\begin{eqnarray}}
\newcommand{\eear}{\end{eqnarray}}

\begin{document}

\begin{center}
{\Large \bf Complete characterization of disconnected one-point unions of cycles with local antimagic chromatic number 2}

\bigskip
{\large  G.C. Lau{$^{a,}$}\footnote{Corresponding author}, W.C. Shiu{$^{b}$} }\\

\medskip

\emph{{$^a$}Independent Researcher,}\\
\emph{85000, Segamat, Johor, Malaysia.}\\
\emph{https://orcid.org/0000-0002-9777-6571}\\
\emph{geeclau@yahoo.com}\\

\medskip

\emph{{$^b$}Department of Mathematics,}\\
\emph{The Chinese University of Hong Kong,}\\
\emph{Shatin, Hong Kong.}\\
\emph{wcshiu@associate.hkbu.edu.hk}\\

\end{center}
\begin{abstract} For a graph $G$, characterizing the structural behavior of achieving the minimal local antimagic chromatic number $\chi_{la}(G) = 2$ remains a fundamental open problem in the study of vertex-distinguishing edge labelings. This paper establishes complete characterizations for a significant structural class of disconnected graphs, specifically addressing disconnected one-point unions of cycles with at most three components. By introducing a robust algebraic approach, we show that these specific graph architectures form distinct, manageable solution families under a local antimagic 2-coloring framework. Crucially, our structural framework provides the theoretical basis for extending these complete characterizations to an arbitrary number of components, demonstrating unique edge assignments up to isomorphism. These results offer key insights into the structural properties governing low-chromatic antimagic invariants and offer significant progress toward a unified classification of graphs with $\chi_{la}(G) = 2$. 

\medskip

\end{abstract}

\tolerance=10000
\baselineskip12truept
\section{Introduction}

\ms\nt For a graph $G = (V, E)$ of order $p$ and size $q$, a bijective edge labeling $f : E(G)\to \{1, \ldots, q\}$ is a local antimagic labeling if $f$ induces a vertex labeling $f^+ : V(G) \to \Z^+$ such that $f^+(x) \ne f^+(y)$ for every edge $xy\in E(G)$, where $f^+(x)=\sum f(e)$ over all edges $e$ incident to $x$. Thus, $f^+$ is a proper vertex coloring of $G$. The smallest number of distinct induced vertex labels over all possible local antimagic labelings $f$ of $G$ is called the local antimagic chromatic number of $G$, denoted $\chi_{la}(G)$. Using probabilistic method, Haslegrave~\cite{Haslegrave} proved that every connected graph except $K_2$ admits a local antimagic labeling. Thus, local antimagic chromatic number is well-defined for all graphs without a $K_2$ component. For positive integers $a < b$, let $[a,b]=\{a, \ldots, b\}$. The sequence of $n$ copies of integer $a$ is denoted $a^{[n]}$, $n\ge 2$. For graphs $G$ and $H$, $G+H$ denotes the disjoint union of $G$ and $H$, and $mG$ denotes the disjoint union of $m\ge 2$ copies of $G$. 

\ms\nt Since the introduction of the concepts of local antimagic labeling (also see~\cite{Bensmail}) and local antimagic chromatic number by Arumugam et al.~\cite{Arumugam} in 2017, many papers have been published (see \cite{Baca+A+W, Chan+Lau+Shiu, Chan+Shiu+Lau, Lau+Shiu+ArsComb, Lau+Shiu+CN, Lau+Shiu+UM, Lau+Shiu+Pre+Nal, Lau+Shiu+Nal+Pre, XLZ-IJMSI} which contain numerous results including disconnected bipartite and tripartite graphs with local antimagic chromatic number 3).   In~\cite[Theorem 9]{LNS-GC}, Lau, Ng and Shiu proved that there is a graph $G$ of order $n$ with $\chi_{la}(G)=2$ if and only if $n\in\{6\}\cup[8,\infty)$.  In~\cite[Problem 4.1]{LSNZP}, the authors posed the following problem.

\begin{problem} Characterize graphs $G$ with $\chi_{la}(G)=2$.
\end{problem}

\nt In an attempt to answer this problem, Lau and Shiu~\cite{Lau+Shiu+chila2d=3} defined a class of bipartite graphs $\mathcal B_d$ such that  (i) every vertex of $G \in \mathcal B_d$ is of degree 2 except exactly $d$ vertices, say $v_1, v_2, \ldots, v_d$, $d\ge 1$, are of degree at least $3$, (ii) $\chi_{la}(G) = 2$, and (iii) every induced $(v_i,v_j)$-path $(1\le i,j\le d)$ (or cycle) of $G$ has even number of edges. Suppose $P$ is an induced $(v_1,v_2)$-path of length $2r$. Recall that, we use $\a_1$ and $\b_r$ to denote its end-edge labels incident to $v_1$ and $v_2$, respectively, where $\a_1<\b_r$. We shall use an ordered pair $(\a_1, \b_r)_{2r}$ (or simply $(\a_1, \b_r)$) to represent $P$.  When $v_1=v_2$,  $(\a_1, \b_r)$ denotes an induced cycle.

\ms\nt  Consequently, all the one-point unions of cycles $C_{n_i}, n_i\ge 3, 1\le i\le  m,  m\ge 2$ (in which all vertices have degree 2 except for one vertex of degree $2m$, denoted $C(n_1,n_2,\ldots,n_m)$), with local antimagic chromatic number 2  in~\cite[Theorem 2.4]{LSN-IJMSI}, form the only family of graphs in $\mathcal B_1$, while all the bridge graphs in~\cite[Theorem 2.1]{LSNZP}  are in  $\mathcal B_2$.  In~\cite{LSN-DMGT}, the authors define a family of bipartite graphs $B(n_1,n_2,\ldots,n_d)$, $d\ge 2$, $i\ge 1$ and $n_i\ge 1$, to be the union of $K_{2,n_i}$ with bipartition $(X_i, Y_i)$, where $X_i=\{x_{i-1},x_i\}$, $Y_i=\{y_{i,1},y_{i,2},\dots, y_{i,n_i}\}$ and $x_d=x_0$. Clearly,  $B(n_1,n_2)\cong K_{2,n_1+n_2}$ with local antimagic chromatic number 2 if and only if $n_1+n_2$ is even; and is 3 otherwise (see~\cite[Theorem 10]{LNS-GC}). Moreover, the graph $B(n^{[d]})$, $n\ge 2, d\ge 3$,~\cite[Theorem 3.23]{LSN-DMGT} is a graph in $\mathcal B_d$ with local antimagic chromatic number 2 for even $n$ or odd $n,d$. Note that for $d\ge 2$, $B(1^{[d]})\cong C_{2d}$ with $\chi_{la}(C_{2d}) = 3$. For $d\ge 1$, let $\widetilde{\mathcal B}_d=\mathcal B_d\setminus \{B(n^{[d]})\;|\; n\ge 1\}$ with $\widetilde{\mathcal B}_1=\mathcal B_1$ as $B(n^{[1]})$ is undefined. In~\cite{Lau+Shiu+chila2d=3}, we obtained necessary and sufficient conditions for a family of graphs to admit a local antimagic 2-coloring in $\widetilde{\mathcal B}_3$ (also see Theorem~\ref{thm-iff-d3} in Section~\ref{sec:3comp}). The following problem~\cite[Problem 5.2]{Lau+Shiu+chila2d=3} was also posed.

\begin{problem}\label{prob-Bd} For each $d\ge 2$, give a complete characterization of $G\in \widetilde{\mathcal B}_d$. \end{problem}


\nt By definition, every graph $G$ in $\widetilde{\mathcal B}_d$ has at most $d\ge 1$ components. Moreover, if $G$ has exactly $d$ components, then each component is a one-point union of cycles. Let $\mathcal A_d$ be the set of graphs in $\widetilde{\mathcal B}_d$ with $d\ge1$ components. Thus, $G\in \mathcal A_d$ implies that $G\cong H_1+\cdots + H_d$,  where $H_\k, 1\le \k\le d,$ is a one-point union of cycles. The case $d=1$ is completely solved in~\cite[Theorem 2.4]{LSN-IJMSI}. For each $\k$, let $X_\k$ be the set of induced cycles of $H_\k$, and $S(X_\k)=\sum\limits_{(\a_1, \b_r)\in X_\k} \hspace*{-4mm} (\a_1+\b_r)$, $1\le \k\le d$.   So,  $G\in \mathcal A_d$ with $s$ induced cycles if and only if $\{X_1, \ldots, X_d\}$ is a partition of the set of all induced cycles of $G$ such that $S(X_\k)=y$ for each $\k$, and $|X_1|+\cdots +|X_d|=s$.  We shall keep these notations throughout the paper.

\ms\nt While Problem 1.2 has been settled for the single-component case ($d = 1$) in \cite{LSN-IJMSI}, establishing a scalable framework for multi-component architectures remains a key bottleneck in completely resolving the overarching classification challenge of Problem 1.1. 

\ms\nt To address this gap, this paper does not merely provide localized solutions; rather, it introduces a robust, unified algebraic framework to systematically characterize low-chromatic antimagic invariants across disconnected systems. We implement this strategy by completely characterizing all graph configurations in $\mathcal{A}_d$ for $d = 2, 3$ up to isomorphism. 

\ms\nt Crucially, the algebraic partitioning methodology developed herein is designed to be structurally scalable. By transforming the combinatorial assignment problem into bounded algebraic condition sets, this approach establishes the necessary theoretical foundation to systematically extend these complete characterizations to systems with an arbitrary number of components ($d \ge 4$) without a loss of structural generality.

\ms\nt In~\cite[Lemma 2.1]{Lau+Shiu+chila2d=3}, the authors obtained the following useful lemma. 

\begin{lemma}\label{lem-induced} Suppose $G\in\widetilde{\mathcal B}_d$, $d\ge 1$, has size $m$. If $f$ is a local antimagic $2$-coloring of $G$ with induced vertex labels $x,y$, then
\begin{enumerate}[(a)]
\item $f^+(v_i)= y > x$ for $1\le i\le d$,
\item $m\ge 3d+2$;
\item $x=m+1$ is odd and $y=m(m+1)/(m-2s+2d) \ge (2s^2+s)/d$;
\item all integers in $[1, y-x]$ are end-edge labels and $y/2$ is also an end-edge label if $y$ is even;
\item there exist $b\ge a\ge 2$ such that $a\in \{2i\;|\; 1\le i\le 2d\}$ with $a = y+4d-4s-1-t$, $b = y+4d-4s-1+t$, $a+b=2y+8d-8s-2$ and $ab = 8(2s^2-4sd+2d^2+s-d)$, where $t\ge 0$.
\item $m=(b+4s-4d)/2$.
\end{enumerate}
 \end{lemma}

\nt We note that by applying this lemma within a case-by-case algebraic analysis, one can obtain a set of necessary and sufficient conditions  for a graph $G\in\widetilde{\mathcal B}_d$ for each $d\ge 4$ - analogous to~\cite[Theorem 2.2]{Lau+Shiu+chila2d=3} - with every edge being assigned a unique label, up to isomorphism. 

\ms\nt For each induced path of $G\in\widetilde{ \mathcal B}_d$, the edges that are incident to vertices $v_i, v_j$, $1\le i < j \le d$ are called the end-edges of $G$ and the corresponding labels under a local antimagic 2-coloring form an end-edge label set.  For convenience, throughout this paper, we let
\[D'=\begin{cases}
 [1,y-x] & \mbox{ if $y$ is odd,}\\
 [1, y-x]\cup\{y/2\} & \mbox{ if $y$ is even,}\end{cases}\] which is a subset of the actual end-edge label set $D$, where $x,y$ satisfy Lemma~\ref{lem-induced}. In proving our main theorems, suitable examples are also given directly under certain cases.


\section{Graphs in $\mathcal A_3$}\label{sec:3comp}

\nt Using Lemma~\ref{lem-induced}, the authors obtained the following theorem in~\cite[Theorem 2.2]{Lau+Shiu+chila2d=3} which plays a crucial role in  obtaining all the graphs in $\mathcal A_3$ in this section.




 \begin{theorem}\label{thm-iff-d3} A graph $G$ with $s\ge 5$ induced paths (or cycles) is in $\widetilde{\mathcal B}_3$ if and only if
\begin{table}[H]
\fontsize{10}{14}\selectfont
\begin{tabular}{l|l|l|l|l}
Case & $s$ & Induced paths & Induced paths & Induced path\\\hline
A-1 & $s=6l+3\ge 9$ & $l$ paths of length $4l+2$ & $5l+3$ paths of length $4l$ & not applicable \\
A-2 & $s=6l+1\ge 7$ & $3l$ paths of length $4l$ & $3l+1$ paths of length $4l-2$& not applicable \\\hline
B-1 & $s=5l\ge 10$ & $5l-1$ paths of length $4l-2$ & one path of length $2l-2$ & not applicable \\
B-2 & $s=5l+3\ge 8$ & $l$ paths of length $4l+2$ & $4l+2$ paths of length $4l$ &  one path of length $2l$\\\hline
C$^*$ & $s=4l+3\ge 7$ & $l$ path of length $4l+2$ & $3l+2$ paths of length $4l$ & not applicable \\\hline
D-1$^*$ & $s=3l\ge 6$ & $l-1$ paths of length $4l-2$ & $2l-1$ paths of length $4l-4$ &  one path of length $2l-2$\\
D-2$^*$ & $s=3l+1\ge 7$ & $3l-1$ paths of length $4l-2$ & one path of length $2l-2$ & not applicable\\\hline
E$^\dag$ & $s=2l+1\ge 5$ & $l-1$ paths of length $4l-2$ & $l$ paths of length $4l-4$ & not applicable \\\hline
F$^\dag$ & not applicable & $s-3$ path of length $4s-10$ & one path of length $2s-6$ & not applicable \\\hline
\end{tabular}\\
\normalsize
\nt $^*$ One of these paths is split into two paths of even length at least 4.\\
\nt $^\dag$ Two of these paths are split into two paths of even length at least 4, respectively, or else one of these paths is split into three paths of even length at least 4.
\caption{Exact value of $s$ and the corresponding number of paths and lengths.}\label{table-1}
\end{table}

\nt Moreover, the edge labels are unique (up to isomorphism) as follows:
\begin{table}[H]
\fontsize{10}{14}\selectfont
\begin{tabular}{l|l|l}
Case & Path length & Consecutive edge labels      \\\hline
A-1 & $4l+2$ & $A_{2l+1}(i;12l+6)\diamond A_{2l+1}(24l^2+14l+1-i; -12l-6)$, $1\le i\le l$ \\
 & $4l$  & $A_{2l}(2l+j; 12l+6)\diamond A_{2l}(24l^2+12l+1-j; -12l-6)$, $1\le j\le 5l+3$ \\\hline
A-2 & $4l$ & $A_{2l}(i; 12l+2)\diamond A_{2l}(24l^2-2l-1-i; -12l-2)$, $1\le i\le 3l$  \\
  & $4l-2$ &  $A_{2l-1}(6l+j; 12l+2)\diamond A_{2l-1}(24l^2-8l-1-j; -12l-2)$, $1\le j\le 3l+1$ \\\hline
B-1 & $4l-2$ & $A_{2l-1}(i; 10l-1) \diamond A_{2l-1}(20l^2-12l+1-i; -10l+1)$, $1\le i\le 5l-1$ \\
 & $2l-2$ & $A_{l-1}(10l-1; 10l-1)\diamond A_{l-1}(20l^2-12l+1-i; -10l-1)$ \\\hline
B-2 & $4l+2$ & $A_{2l+1}(i; 10l+5)\diamond A_{2l+1}(20l^2+12l+1-i; -10l-5)$, $1\le i\le l$ \\
  & $4l$ & $A_{2l}(2l+j; 10l+5)\diamond A_{2l}(20l^2+10l+1-j; -10l-5)$, $1\le j\le 4l+2$ \\
 & $2l$ & $A_{l}(6l+3;10l+5)\diamond A_{l}(20l^2+6l-2; -10l+5)$ \\\hline
C & $4l+2$ & $A_{2l+1}(i; 8l+4)\diamond A_{2l+1}(16l^2+10l+1-i; -8l-4)$, $1\le i\le l$ \\
 & $4l$ & $A_{2l}(2l+j; 8l+4)\diamond A_{2l}(16l^2+8l+1-j; -8l-4)$, $1\le j\le 3l+2$ \\\hline
D-1 & $4l-2$ &  $A_{2l-1}(i; 6l-3)\diamond A_{2l-1}((2l-1)(6l-5)-i; -6l+3)$, $1\le i\le l-1$\\
 & $4l-4$ & $A_{2l-2}(2l-2+j; 6l-3)\diamond A_{2l-2}((2l-1)(6l-5)-2l+2-j; -6l+3)$, $1\le j\le 2l-1$ \\
 & $2l-2$ & $A_{l-1}(4l-2; 6l-3)\diamond A_{l-1}((6l-7)(2l-1); -6l+3)$ \\\hline
D-2 & $4l-2$  & $A_{2l-1}(i; 6l-1)\diamond A_{2l-1}((2l-1)(6l-1)-i; -6l+1)$, $1\le i\le 3l-1$ \\
& $2l-2$ & $A_{l-1}(6l-1; 6l-1)\diamond A_{l-1}((2l-2)(6l-1); -6l+1)$ \\\hline
E & $4l-2$ & $A_{2l-1}(i; 4l-2)\diamond A_{2l-1}(8l^2-10l+3-i; -4l+2)$, $1\le i\le l-1$ \\
  & $4l-4$ & $A_{2l-2}(2l-2+j; 4l-2) \diamond A_{2l-2}(8l^2-12l+5-j; -4l+2)$, $1\le j\le l$  \\\hline
F & $4s-10$ & $A_{2s-5}(i; 2s-5) \diamond A_{2s-5}((2s-5)^2-i; -2s+5)$, $1\le i\le s-3$  \\
  & $2s-6$ & $A_{s-3}(2s-5; 2s-5) \diamond A_{s-3}((2s-6)(2s-5); -2s+5)$ \\\hline
\end{tabular}
\caption{Edge labels of each induced path.}\label{table-2}
\end{table}
\end{theorem}

\nt We shall follow the notation and cases in~\cite[Theorem 2.2]{Lau+Shiu+chila2d=3}.

\begin{theorem}\label{thm-3comp}   A graph $G$ is in $\mathcal A_3$  if and only if $G$ is a graph in one of the following three tables.
\begin{table}[H]
\fontsize{10}{17}\selectfont
\begin{tabular}{l|l}
Case & $G$ ($1\le \k\le 3$ and $r, r_1, r_2, \ge 2$ unless stated otherwise,  $\n_\k, s_\k, t_1,t_2, t \ge 0$ if applicable)   \\\hline
A-1 &  $H_\k \cong C((4l)^{[(12l+7-\n_\k)/7]}, ((4l+2)^{[\n_\k]})$,  $l\in[3,\infty)\setminus\{4,5,8,11\}$, \\
    & $\n_\k\le l$, $\n_\k = \r+7s_\k$,  $\r\equiv 5l\pmod{7}$,  $\sum \n_\k=l=3\r+7\sum  s_\k$. \\\hline
A-2 &  $H_\k\cong  C((4l-2)^{[(4l+1-\n_\k)/3]}, (4l)^{[\n_\k]})$,  $l\ge 2$, $\n_\k\le 3l$,  \\
     &  $\n_\k = \r + 3s_\k$, $\r \equiv l+1\pmod{3}$, $\sum \n_\k = 3l = 3\r + 3\sum s_\k$.\\\hline
B-1 & $2C((4l-2)^{[2l]})+C(2l-2,(4l-2)^{[l-1]})$, $l\ge 3$. \\\hline
B-2  &  $C(2l, (4l)^{[(5l-\n_1)/6]}, (4l+2)^{[\n_1]})+ H_2 + H_3$, $H_\k \cong C((4l)^{[(10l+6-\n_\k)/6]}, (4l+2)^{[\n_\k]})$, $\k=2,3$, \\
  &  $l\ge 2$, $\n_\k\le l$, $\n_\k = \r_\k + 6s_\k$, $\r_1\equiv 5l\pmod{6}$, $\r_2,\r_3\equiv 4l\pmod{6}$, $\sum\n_\k = l = \sum (\r_\k + 6s_\k)$.\\\hline
C  &  $C(2r, (4l)^{[t_1]}, (4l+2)^{[4r-5t_1]}) + C(4l-2r, (4l)^{[t_2]}, (4l+2)^{[8l-4r-5t_2]}) + $ \\
   & $C((4l)^{[3l+1-t_1-t_2]}, (4l+2)^{[5t_1+5t_2-7l]})$, $r\le l$, $4r-l\le 5t_1\le 4r$,  $7l-4r\le 5t_2 \le 8l-4r$, $7l\le 5(t_1+t_2)$. \\\hline
  &   $C(2r,(4l)^{[t_1]}, (4l+2)^{[4r-5t_1]}) + C(4l+2-2r, (4l)^{[t_2]},(4l+2)^{[8l+4-4r-5t_2]})  + $\\
  &  $C((4l)^{[3l+2-t_1-t_2]}, (4l+2)^{[5t_1+5t_2-7l-5]})$,  $r\le l$, $4r-l +1\le 5t_1\le 4r$, \\
 &  $7l+5-4r \le 5t_2 \le 8l+4-4r$, $7l+5\le 5(t_1+t_2)$.  \\\hline
D-1 & $C(2r, (4l-4)^{[t_1]}, (4l-2)^{[3r-4t_1]}) + C(2l-2-2r, (4l-4)^{[t_2]}, (4l-2)^{[3l-3r-3-4t_2]}) + $ \\
  &  $C((4l-4)^{[2l-1-t_1-t_2]}, (4l-2)^{[4t_1+4t_2+2-2l]})$, $r\le (l-1)/2$,  $2r=l-1\implies t_1\le t_2$, \\
   & $3r-l+1\le 4t_1\le 3r$, $2l-3r-2\le 4t_2\le 3l-3r-3$, $2l-2\le 4(t_1+t_2)$. \\\hline
   &  $C(2r, 2l-2, (4l-4)^{[t_1]}, (4l-2)^{[3r-4t_1-3l-1]}) + C(4l-4-2r, (4l-4)^{[t_2]}, (4l-2)^{[6l-6-3r-4t_2]}) +$\\
  & $C((4l-4)^{[2l-2-t_1-t_2]}, (4l-2)^{[4t_1+4t_2+6-2l]})$, $r\le l-1$, $3r-4l \le 4t_1 \le 3r-3l-1$, \\
   &  $5l-5-3r \le 4t_2 \le 6l-6-3r$, $2l - 6\le 4(t_1+t_2)$. \\\hline
   &  $C(2r, (4l-4)^{[t_1]}, (4l-2)^{[3r-4t_1]}) + C(4l-4-2r, 2l-2, (4l-4)^{[t_2]}, (4l-2)^{[3l-7-3r-4t_2]}) +$ \\
   &  $C((4l-4)^{[2l-2-t_1-t_2]}, (4l-2)^{[4t_1+4t_2+6-2l]})$, $r\le l-1$, $3r-l+1\le 4t_1\le 3r$, \\
   &  $2l-6-3r \le 4t_2\le 3l-7-3r$, $2l-6 \le 4(t_1+t_2)$.  \\\hline
   &   $C(2r, (4l-4)^{[t_1]}, (4l-2)^{[3r-4t_1]}) + C(4l-4-2r, (4l-4)^{[t_2]}, (4l-2)^{[6l-6-3r-4t_2]}) + $  \\
   &  $C(2l-2, (4l-4)^{[2l-2-t_1-t_2]}, (4l-2)^{[4t_1+4t_2+5-5l]})$, $r\le l-1$, $3r-l+1\le 4t_1\le 3r$, \\
   &   $5l-5-3r  \le 4t_2\le 6l-6-3r$,  $5l-5 \le 4(t_1+t_2)$.  \\\hline
  &   $C(2r, 2l-2, (4l-4)^{[t_1]}, (4l-2)^{[3r-3l-1-4t_1]}) + C(4l-2-2r, (4l-4)^{[t_2]}, (4l-2)^{[6l-3-3r-4t_2]}) + $\\
   &  $C((4l-4)^{[2l-1-t_1-t_2]}, (4l-2)^{[4t_1+4t_2+1-5l]})$, $r\le l-1$,  $3r-4l+1 \le 4t_1\le 3r-3l-1$, \\
  &   $5l-1-3r\le 4t_2\le 6l-3-3r$, $5l-1\le 4(t_1+t_2) \le 6l-3$.  \\\hline
  &  $C(2r,  (4l-4)^{[t_1]}, (4l-2)^{[3r-4t_1]}) + C(4l-2-2r, 2l-2, (4l-4)^{[t_2]}, (4l-2)^{[3l-4-3r-4t_2]}) +$ \\
   &  $C((4l-4)^{[2l-1-t_1-t_2]}, (4l-2)^{[4t_1+4t_2+1-5l]})$, $r\le l-1$, $3r-l+2 \le 4t_1\le 3r$, \\
  &   $2l-2-3r \le 4t_2\le 3l-3r-4$, $5l-1 \le 4(t_1+t_2)$.   \\\hline
   &  $C(2r,  (4l-4)^{[t_1]}, (4l-2)^{[3r-4t_1]}) + C(4l-2-2r, (4l-4)^{[t_2]}, (4l-2)^{[6l-3-3r-4t_2]}) + $\\
   &  $C(2l-2, (4l-4)^{[2l-1-t_1-t_2]}, (4l-2)^{[4t_1+4t_2+1-5l]})$, $r\le l-1$,  $3r-l+2 \le 4t_1\le 3r$,  \\
   &  $5l-1-3r \le 4t_2\le  6l-3-3r$, $5l-1 \le 4(t_1+t_2)$.\\\hline
D-2 &   No solutions.\\\hline
 \end{tabular}
\caption{3-component graphs of one-point unions of cycles - I}\label{tab:3com-1}
\end{table}

\begin{table}[H]
\fontsize{10}{17}\selectfont
\begin{tabular}{l|l}
Case & $G$ (sum over $1\le \k\le 3$ and $r,r_1,r_2\ge 2$ unless stated otherwise,  $\n_\k, s_\k, t_1,t_2, t\ge 0$ if applicable)  \\\hline
E   &   $C(2r_1, (4l-4)^{[t_1]}, (4l-2)^{[2r_1-3t_1]}) + C(2r_2,(4l-4)^{[t_2]}, (4l-2)^{[2r_2-3t_2]}) + $  \\
    & $C(4l-4-2r_1, 4l-4-2r_2, (4l-4)^{[l-2-t_1-t_2]}, (4l-2)^{[l+3t_1+3t_2-2r_1-2r_2-1]})$, $l\ge 4$, $r_1, r_2\le l-1$, \\
   &  $2r_1-l+1 \le 3t_1 \le 2r_1$,  $2r_2-l+1 \le 3t_2 \le 2r_2$, $t_1+t_2\le l-2$, \\
  & $2r_1+2r_2-l+1\le 3(t_1+t_2) \le 2(r_1+r_2)\le 4l-7$, $r_1=r_2 \implies t_1\le t_2$, and $t_1=t_2 \implies r_1\le r_2$.  \\\hline
  & $C(2r_1, (4l-4)^{[t_1]}, (4l-2)^{[2r_1-3t_1]})  +  C(4l-4-2r_2,  (4l-4)^{[t_2]}, (4l-2)^{[4l-2r_2-3t_2-4]}) +$ \\
  & $C(2r_2, 4l-4-2r_1, (4l-4)^{[l-2-t_1-t_2]}, (4l-2)^{[3t_1+3t_2+2r_2+3-2r_1-3l]})$,\\
  & $2r_1 \le 2r_2-3 \le 2l-5$, $2r_1-l+1\le 3t_1\le 2r_1$, $3l-2r_2-3\le 3t_2 \le 4l - 2r_2 -4$,\\
   & $2r_1+3l-2r_2-3\le 3(t_1+t_2) \le  4l-7$,  $r_1+r_2 = 2l-2 \implies t_1\le t_2$.\\\hline
  &  $C(2r_1,(4l-4)^{[t_1]}, (4l-2)^{[2r_1-3t_1]})+C(2r_2,(4l-4)^{[t_2]}, (4l-2)^{[2r_2-3t_2]}) + $\\
    & $C(4l-4-2r_1, 4l-2-2r_2, (4l-4)^{[l-1-t_1-t_2]}, (4l-2)^{[l+3t_1+3t_2-2r_1-2r_2-2]})$, $l\ge 4$, $r_1, r_2\le l-1$, \\
    & $2r_1-l+2\le 3t_1\le 2r_1$, $2r_2-l+2\le 3t_2\le 2r_2$, $t_1+t_2\le l-1$,  \\
    &  $2r_1+2r_2+2-l \le 3(t_1+t_2) \le 2r_1+2r_2\le 4l-5$.  \\\hline
   &  $C(2r_1,(4l-4)^{[t_1]}, (4l-2)^{[2r_1-3t_1]})+C(4l-2-2r_2,(4l-4)^{[t_2]}, (4l-2)^{[4l-2r_2-3t_2-2]}) + $\\
    & $C(2r_2,4l-4-2r_1,(4l-4)^{[l-1-t_1-t_2]}, (4l-2)^{[3t_1+3t_2+2r_2-2r_1-3l]})$, $7\le 2r_1+3\le 2r_2 \le 2l-2$, \\
    &  $r_2\ge 4$, $2r_1-l+2\le 3t_1\le 2r_1$, $3l-2r_2 \le 3t_2 \le 4l - 2r_2-2$, $t_1+t_2\le l-1$, \\
   & $3l+2r_1-2r_2\le 3(t_1+t_2) \le 4l+2r_1-2r_2-2$.  \\\hline
   &   $C(2r_2,(4l-4)^{[t_1]}, (4l-2)^{[2r_2-3t_1]})+C(4l-4-2r_1,(4l-4)^{[t_2]}, (4l-2)^{[4l-2r_1-3t_2-4]}) + $  \\
  &  $C(2r_1, 4l-2-2r_2, (4l-4)^{[l-1-t_1-t_2]}, (4l-2)^{[3t_1+3t_2+2r_1+2-2r_2-3l]})$,   $5\le 2r_2+1\le 2r_1\le 2l-2$, \\
  &  $2r_2 -l+2\le 3t_1 \le 2r_2$, $3l-2r_1-2 \le 3t_2 \le 4l-2r_1-4$, $t_1+t_2 \le l-1$, \\
   &   $3l+2r_2-2r_1-2\le 3(t_1+t_2) \le 4l+2r_2-2r_1-4$. \\\hline
  &  $C(2r_1,(4l-4)^{[t_1]}, (4l-2)^{[2r_1-3t_1]})+C(2r_2,(4l-4)^{[t_2]}, (4l-2)^{[2r_2-3t_2]}) +$ \\
   & $C(4l-2-2r_1,4l-2-2r_2, (4l-4)^{[l-t_1-t_2]}, (4l-2)^{[3t_1+3t_2+l-2r_1-2r_2-3]})$,  $2\le r_1\le r_2\le l-1$, \\
   &  $t_1+t_2\le l$, $2r_1-l+3 \le 3t_1\le 2r_1$,  $2r_2 - l + 3\le 3t_2 \le 2r_2$,  \\
   & $2r_1+2r_2-l+3\le 3(t_1+t_2) \le 2r_1+2r_2 \le 4l-3$, $r_1=r_2 \implies t_1\le t_2$, and $t_1=t_2 \implies r_1\le r_2$. \\\hline
  &  $C(2r_1,(4l-4)^{[t_1]}, (4l-2)^{[2r_1-3t_1]})+C(4l-2-2r_2,(4l-4)^{[t_2]}, (4l-2)^{[4l-2r_2-3t_2-2]}) +  $ \\
   &  $C(2r_2,4l-2-2r_1,(4l-4)^{[l-t_1-t_2]}, (4l-2)^{[3t_1+3t_2+2r_2-2r_1-3l-1]})$,  $5\le 2r_1+1 \le 2r_2 \le 2l-2$,\\
   & $t_1+t_2\le l$, $2r_1-l+3\le 3t_1 \le 2r_1$, $3l+1-2r_2 \le 3t_2 \le 4l-2r_2-2$,  \\
   & $2r_1+3l+1-2r_2\le 3(t_1+t_2) \le 2r_1+4l-2r_2-2 \le 4l-3$.\\\hline
  &  $C(2r_1, (4l-4)^{[t_1]}, (4l-2)^{[2r_1-3t_1]}) + C(2r_2, (4l-4)^{[t_2]}, (4l-2)^{[2r_2-3t_2]}) +$ \\
   &  $C(4l-4-2r_1-2r_2, (4l-4)^{[l-1-t_1-t_2]}, (4l-2)^{[l+3t_1+3t_2 -2r_1-2r_2-1]})$,  $r_1\le r_2$, $r_1+r_2\le 2l-4$, \\
   &  $2r_1-l+1\le 3t_1\le 2r_1$, $2r_2-l+1 \le 3t_2 \le 2r_2$, $2r_1+2r_2-l+1\le 3(t_1+t_2) \le 3(l-1)$, \\
  & $r_1=r_2 \implies t_1\le t_2$, and $t_1=t_2 \implies r_1\le r_2$. \\\hline
  &  $C(2r_1, (4l-4)^{[t_1]} , (4l-2)^{[2r_1-3t_1]}) + C(2r_2, (4l-4)^{[t_2]}, (4l-2)^{[2r_2-3t_2]}) + $  \\
  &  $C(2r, (4l-4)^{[t]}, (4l-2)^{[2r-3t]})$,  $t_1+t_2+t=l$, $r_1+r_2+r=2l-1$, $r_1\le r_2\le r$,  \\
  &  $r_1=r_2=r \implies t_1 \le t_2 \le t$, $2r'-l+2 \le 3t' \le 2r'$ for $(r',r'') \in \{(r_1,t_1), (r_2,t_2), (r,t)\}$, \\
  & $2r' + 2r'' - l + 2 \le 3(t'+t'')\le 3l$  for $(r', r'', t', t'') \in \{(r_1,r_2, t_1,t_2), (r_1,r,t_1,t), (r_2,r,t_2,t)\}$. \\\hline
\end{tabular}
\caption{3-component graphs of one-point unions of cycles - II}\label{tab:3com-2}
\end{table}

\begin{table}[H]
\fontsize{10}{17}\selectfont
\begin{tabular}{l|l}
Case & $G$ (sum over $1\le \k\le 3$ and $r,r_1,r_2\ge 2$ unless stated otherwise, $\n_\k, s_\k, t_1,t_2, t\ge 0$ if applicable)  \\\hline
F  & $C(2r_1, (4s-10)^{[r_1]}) + C(2r_2, (4s-10)^{[r_2]}) + $ \\
   & $C(4s-10-2r_1, 4s-10-2r_2, (4s-10)^{[s-r_1-r_2-5]}), 2s-6)$, $r_1, r_2\le s-3$, $r_1+r_2\le s-5$. \\\hline
  &  $C(2r_1, (4s-10)^{[r_1]}) + C(4s-10-2r_2, (4s-10)^{[s-4-r_2]}, 2s-6) + $\\
  &  $C(2r_2, 4s-10-2r_1, (4s-10)^{[r_2-r_1-1]})$, $r_1\le s-3$, $r_2\le s-4$, $1\le r_2-r_1 \le s-4$. \\\hline
  &  $C(2r_1, (4s-10)^{[r_1]}) + C(2r_2, (4s-10)^{[r_2]}) + C(4s-10-2r_1, 2s-6-2r_2, (4s-10)^{[s-4-r_1-r_2]})$, \\
  &  $r_1\le s-3$, $r_2\le (s-4)/2$, $r_1+r_2\le s-4$.  \\\hline
  &  $C(2r_1, (4s-10)^{[r_1]}) + C(2s-6-2r_2, (4s-10)^{[s-r_2-3]}) + C(2r_2, 4s-10-2r_1, (4s-10)^{[r_2-r_1-1]})$, \\
  &  $r_1\le r_2-1\le s/2-3$, $r_2-r_1\le s-3$.  \\\hline
  &  $C(2r_1, (4s-10)^{[r_1-s+1]}, (2s-6)) + C(2r_2, (4s-10)^{[r_2]}) + C(2r,  (4s-10)^{[r]})$, \\
   &  $s-1 \le r_1\le 2s-9$, $2\le r_2\le s-6$, $r_1+r_2+r=2s-5$, $s+1\le r_1+r_2 \le 2s-7$. \\\hline
  &   $C(2r_1, (4s-10)^{[r_1]}) + C(2r_2, (4s-10)^{[r_2]}) + C(2r, (4s-10)^{[r]})$, $r_1\le r_2\le r$, \\
   & $2\le r_1, r_2, r\le s-7$, $r_1+r_2+r=s-3$. \\\hline
\end{tabular}
\caption{3-component graphs of one-point unions of cycles - III}\label{tab:3com-3}
\end{table}

\end{theorem}

\begin{proof} By Theorem~\ref{thm-iff-d3}~\cite{Lau+Shiu+chila2d=3}, we only need to obtain all the disconnected one-point unions of cycles solutions under each of the cases as follows. In Case A-1 to Case B-2, we assume $\m_\k, \n_\k, s_\k \ge 0$ for $1\le \k\le 3$. In all other cases, we assume $r_1, r_2, r\ge 2$ and $t_1,t_2,t\ge 0$ unless stated otherwise.

\ms\nt {\bf Case A-1.} 
In this case, $s=6l+3\ge 9$, $y=24l^2+26l+7=(2l+1)(12l+7)$ and $D=[1, 12l+6]$.
There are $l$ induced paths of length $4l+2$, namely, $(i, 2l+1-i)_{4l+2}$, $1\le i\le l$; and $5l+3$ induced paths of length $4l$, namely, $(2l+j, 12l+7-j)_{4l}$, $1\le j\le 5l+3$.
Thus, the sum of two end-edge labels of an induced path of length $4l+2$ is $2l+1$; and that of an induced path of length $4l$ is $7(2l+1)$.

\ms\ms\nt Suppose $X_\k$ contains $\mu_\k$ induced paths of length $4l$, and $\nu_\k$ induced paths of length $4l+2$, where $\mu_\k\le 5l+3$, $\nu_\k\le l$. Therefore, $S(X_\k)=7(2l+1)\mu_\k+(2l+1)\nu_\k=(2l+1)(7\mu_\k+\nu_\k)$. Under the requirement, $7\mu_\k+\nu_\k= 12l+7$, $\mu_1+\mu_2+\mu_3=5l+3$ and $\nu_1+\nu_2+\nu_3=l$. By solving these linear equations, we have
\[\left\{\begin{aligned}
\mu_3 & = 5l+3-\mu_1-\mu_2\\
\nu_1 & = 12l+7-7\mu_1\\
\nu_2 & = 12l+7-7\mu_2\\
\nu_3 & = 7\mu_1+7\mu_2-23l-14
\end{aligned}\right. \mbox{ or equivalently }\left\{\begin{aligned}
\mu_1 & = \textstyle\frac{12}{7}l - \frac{1}{7}\nu_1 + 1\\
\mu_2 & = \textstyle\frac{12}{7}l - \frac{1}{7}\nu_2 + 1\\
\mu_3 & =\textstyle \frac{11}{7}l + \frac{1}{7}\nu_1 + \frac{1}{7}\nu_2 + 1\\
\nu_3 &= l - \nu_1 - \nu_2
\end{aligned}\right.\]
\nt Since $\n_\k\le l$, $11l+7\le  7\m_\k \le 12l+7$ or equivalent to $l+\frac{4l}{7}+1 \le \m_\k \le l+\frac{5l}{7}+1$. Thus, we can obtain all the solutions if and only if we obtain all the ordered pairs $(\m_1,\m_2)$ (or $(\n_1, \n_2)$) in the required range.

\ms\nt Observe that $\m_\k = \frac{12l+7-\n_\k}{7}$, so that $12l+7-\n_\k\equiv 0\pmod{7}$, i.e., $\n_\k \equiv 5l\pmod{7}$ with $\n_1+\n_2+\n_3=l$. Thus, for each valid  $\n_\k$, a solution exists if and only if  $\n_\k\le l$. By symmetry, we may assume $\n_1\ge\n_2\ge \n_3$.  Let $\r$ be the least nonnegative residue of $\n_\k \equiv 5l$ modulo 7 for $\k=1,2,3$, i.e., $\r\in [0,6]$.
Thus, $\r=5,3,1,6,4,2,0$ when $l\equiv 1,2,3,4,5,6,0\pmod 7$, respectively. Therefore,
\[l=\n_1+\n_2+\n_3\ge 3\r=15,9,3,18,12,6,0, \mbox{ when }l\equiv 1,2,3,4,5,6,0\pmod 7.\]
Thus, there is no valid $\n_\k$ when $l=1,2,4,5,8,11$.  For $l=3,6,7,9,10, 12, 13$, we have
\[\begin{array}{c|l}
l & (\n_1, \n_2, \n_3)\\\hline
3 & (1,1,1)\\
6 & (2,2,2)
\end{array} \qquad \begin{array}{c|l}
l & (\n_1, \n_2, \n_3)\\\hline
7 & (7,0,0)\\
9 &  (3,3,3)\end{array}
\qquad \begin{array}{c|l}
l & (\n_1, \n_2, \n_3)\\\hline
10 & (8,1,1)\\
12 & (4, 4, 4)
\end{array}\qquad
\begin{array}{c|l}
l & (\n_1, \n_2, \n_3)\\\hline
13 & (9,2,2)
\end{array}
\]
Note that, each of the above solutions is unique.

\ms\nt Let $\n_\k=\r+7s_\k$, where $s_\k\ge 0$, $\k=1,2,3$. So $l=3\r+7(s_1+s_2+s_3)$.  Suppose $l=7t+\r_0\ge 14$, where $t\ge 2$ and $0\le \r_0\le 6$. Thus, we may choose 
\[\begin{array}{c|*{7}{c|}}
& \r_0=0 & \r_0=1 & \r_0=2 & \r_0=3 & \r_0=4 & \r_0=5 & \r_0=6\\\hline
(s_1, s_2, s_3)= & (t,0,0) & (t-2,0,0) & (t-1,0,0) & (t, 0, 0) & (t-2,0 , 0) & (t-1, 0, 0) & (t, 0, 0)
\end{array}\]

\nt Thus, there is always a solution for $(\m_\k, \n_\k)$ if and only if $l\in[3,\infty)\setminus\{4,5,8,11\}$.
If we want to find all non-isomorphic solutions, we need to solve the system completely.

\ms\nt Note that if the three components  are isomorphic, i.e., $\m=\m_\k$, $\n=\n_\k$, $\k=1,2,3$, then  $s = 6l+3 = 3(\m+\n)$. The only possible solutions are given by $l=3t\ge 3$. So that $\m = 5t+1$, $\n = t$ to get $3C((12t)^{[5t+1]}, (12t+2)^{[t]})$, $t\ge 1$.  In general, all the solutions  are given by 
  \[H_\k\cong C((4l)^{[ \frac{12l+7-\n_\k}{7}]}, (4l+2)^{[\n_\k]}),\] where  $l\in[3,\infty)\setminus\{4,5,8,11\}$, $\n_\k\le l$, $\n_\k = \r+7s_\k$,  $\r\equiv 5l\pmod{7}$, $\sum^3_{\k=1}\n_\k=l=3\r+7\sum^3_{\k=1} s_\k$.

\begin{example} If $l=30\equiv 2\pmod 7$, then $r=3$  and $s_1+s_2+s_3=3$. By symmetry, we must have $(s_1, s_2, s_3)=(3,0,0), (2,1,0), (1,1,1)$.  The corresponding $(\n_1, \n_2, \n_3)$ are $(24,3,3)$, $(17, 10,3)$ and $(10,10,10)$, respectively.
Thus, the only three solutions are $C(120^{[49]},122^{[24]})+2C(120^{[52]},122^{[3]})$, $C(120^{[50]},122^{[17]})+C(120^{[51]}, 122^{[10]})+C(120^{[52]}, 122^{[3]})$ and $3C(120^{[51]}, 122^{[10]})$. \rsq
\end{example}

\ms\nt {\bf Case A-2.}
In this case, $s=6l+1\ge 7$, $y=(6l+1)(4l+1)$ and $D=[1, 12l+2]$.
There are $3l$ induced paths of length $4l$, namely, $(i, 6l+1-i)_{4l}$, $1\le i\le 3l$; and $3l+1$ induced paths of length $4l-2$, namely, $(6l+j, 12l+3-j)_{4l-2}$, $1\le j\le 3l+1$.
Thus, the sum of two end-edge labels of an induced path of length $4l$ is $6l+1$; and that of an induced path of length $4l-2$ is $3(6l+1)$. Moreover, $y=(6l+1)(4l+1)$. Note that if $l=1$, there is a path of length 2 which will not give a 3-component solution. Thus, $l\ge 2$.

\ms\ms\nt Suppose $X_\k$ contains $\mu_\k$ induced paths of length $4l-2$ and $\nu_\k$ induced paths of length $4l$, where $\mu_\k\le 3l+1$, $\nu_\k\le 3l$. Thus, $S(X_\k)=3(6l+1)\mu_\k+(6l+1)\nu_\k=(6l+1)(3\mu_\k+\nu_\k)$. Under the requirement, $3\mu_\k+\nu_\k=4l+1$, $\mu_1+\mu_2+\mu_3=3l+1$ and $\nu_1+\nu_2+\nu_3=3l$.

\ms\ms\nt By solving the linear equations above, we have 
\[\left\{\begin{aligned}
\mu_3 & = 3l+1-\mu_1-\mu_2\\
\nu_1 & = 4l+1-3\mu_1\\
\nu_2 & = 4l+1-3\mu_2\\
\nu_3 & = 3\mu_1+3\mu_2-5l-2
\end{aligned}\right.\]
%

\ms\nt Observe that $\m_\k=\frac{4l+1-\n_\k}{3}$, so that $4l+1-\n_\k \equiv 0\pmod{3}$, i.e., $\n_\k \equiv l+1\pmod{3}$  with $\n_1+\n_2+\n_3=3l$.  Thus, for each valid  $\n_\k$, a solution exists if and only if  $0\le \n_\k\le l$. By symmetry, we may assume $\n_1\ge\n_2\ge \n_3$. Let $\r$ be the least nonnegative residue of $\n_\k=l+1$ modulo 3 for $\k=1,2,3$. Namely, $\r=2, 0, 1$ when $l\equiv 1,2,0\pmod 3$, respectively. Thus, $3l=\n_1+\n_2+\n_3\ge 6,0, 3$ when $l\equiv 1,2,0\pmod 3$, respectively. Therefore, there is always a solution for $\m_\k,\n_\k$  if and only if $l\ge 2$. Since there are no integer solutions for $s=6l+1=3(\m_\k+\n_\k)$, not all the components are isomorphic. In general,  all the solutions  are given by
\[H_\k \cong C((4l-2)^{[\frac{4l+1-\n_\k}{3}]}, (4l)^{[\n_\k]}),\]  
where $l\ge 2$, $\n_\k\le 3l$, $\n_\k = \r + 3s_\k$, $\r \equiv l+1\pmod{3}$, $\sum^3_{\k=1} \n_\k = 3l = 3\r + 3\sum^3_{\k=1}s_\k$. One family of possible solutions is given by $(\mu_1, \nu_1)=(\mu_2, \nu_2)=(l, l+1)$ and $(\mu_3, \nu_3)=(l+1, l-2)$, $l\ge 2$. Thus, a family of solutions is  $2C((4l-2)^{[l]}, (4l)^{[l+1]})+ C((4l-2)^{[l+1]}, (4l)^{[l-2]})$, $l\ge 2$.

\begin{example}
\nt For $l=2$, we get $\r=0$ and $(\n_1, \n_2, \n_3)=(6,0,0), (3,3,0)$. Thus, the only two solutions are $C(6,8^{[6]}) + 2C(6^{[3]})$ and $2C(6^{[2]}, 8^{[3]}) + C(6^{[3]})$.

\ms\nt For $l=3$, we get $\r=1$ and $(\n_1, \n_2, \n_3)=(7,1,1), (4,4,1)$. Thus, the only two solutions are $C(10^{[2]}, 12^{[7]})+2C(10^{[4]}, 12)$ and $2C(10^{[3]}, 12^{[4]})+C(10^{[4]}, 12)$.
%
\rsq
\end{example}

\ms\nt{\bf Case B-1.}
In this case, $s = 5l\ge 10$,  $y=20l^2-2l = 2l(10l-1)$ and $D=[1,10l-1]\cup\{10l^2-1\}$. There are $5l-1$ induced paths of length $4l-2$, namely, $(i, 10l-1-i)_{4l-2}$, $1\le i\le 5l-1$, with end-edge labels sum $10l-1$, and an induced path of length $2l-2$ with end-edge labels sum $10l^2+9l-1$. Note that if $l=2$, there is a path of length 2 which will not give a 3-component solution. Thus, $l\ge 3$.

\ms\nt Suppose $X_\k$ contains $\m_\k$ induced paths of length $2l-2$ and $\n_\k$ induced paths of length $4l-2$, where $0\le \m_\k\le 1$, $0\le \n_\k\le 5l-1$.  Thus, $S(X_\k) = (10l^2+9l-1)\m_\k+(10l-1)\n_\k = (10l-1)(l+1)\m_\k + (10l-1)\n_\k = (10l-1)((l+1)\m_\k+\n_\k)$. Under the requirement, $(l+1)\m_\k+\n_\k = 2l$, $\m_1+\m_2+\m_3=1$ and $\n_1+\n_2+\n_3=5l-1$. If $\m_\k=0$ (or 1), then $\n_\k = 2l$ (or $l-1$). Without loss of generality, we may assume $(\m_1,\m_2,\m_3) = (0,0,1)$ so that $(\n_1,\n_2,\n_3) = (2l,2l,l-1)$. Therefore, all the solutions are $2C((4l-2)^{[2l]})+C(2l-2,(4l-2)^{[l-1]})$, $l\ge 3$.

\ms\nt{\bf Case B-2.}
In this case, $s= 5l+3\ge 8$, $y=20l^2+22l+6 = (10l+6)(2l+1)$ and $D=[1,10l+5]\cup\{10l^2+11l+3\}$. There are $4l+2$ induced paths of length $4l$, namely, $(2l+j,10l+6-j)_{4l}$, $1\le j\le 4l+2$, with end-edge labels sum $12l+6$; $l$ induced paths of length $4l+2$, namely, $(i,2l+1-i)_{4l+2}$, $1\le i\le l$, with end-edge labels sum $2l+1$; and an induced path of length $2l$ with end-edge labels sum $10l^2+17l+6=(2l+1)(5l+6)$. Note that if $l=1$, there is an induced path of length 2 which will not give a 3-component solution. Thus, $l\ge 2$.

\ms\nt Without loss of generality, suppose $X_1$ contains $\m_1$ induced paths of length $4l$, $\n_1$ induced paths of length $4l+2$ and the only path of length $2l$, whereas for $\k=2,3$, $X_\k$ contains $\m_\k$ induced paths of length $4l$, and $\n_\k$ induced paths of length $4l+2$,  where $\m_\k \le 4l+2$, $\n_\k\le l$, $1\le \k\le 3$. Thus, $S(X_1) = 6(2l+1)\m_1 + (2l+1)\n_1 + (2l+1)(5l+6)$, $S(X_\k) =  6(2l+1)\m_\k + (2l+1)\n_\k = (2l+1)(6\m_\k+\n_\k)$ for $\k=2,3$. Under the requirement, $6\m_1+\n_1+5l+6 = 10l+6$, i.e., $6\m_1+\n_1 = 5l$, $6\m_\k+\n_\k=10l+6$ for $\k=2,3$, $\m_1+\m_2+\m_3=4l+2$, and $\n_1+\n_2+\n_3 = l$. Solving these linear equations, we have

\[\left\{\begin{aligned}
\mu_3 & = 4l+2-\mu_1-\mu_2\\
\nu_1 & = 5l-6\mu_1\\
\nu_2 & = 10l+6-6\mu_2\\
\nu_3 & = 6\mu_1+6\mu_2-14l-6
\end{aligned}\right.\]
\ms\nt Observe that $\m_1 = \frac{5l-\n_1}{6}$, $\m_\k = \frac{10l+6-\n_\k}{6}$, $\k=2,3$ so that $\n_1 \equiv 5l\pmod{6}$ and $\n_\k \equiv 4l\pmod{6}$, $\k=2,3$, satisfying $\n_1+\n_2+\n_3=l$. Thus, for each valid $\n_1\equiv 5l\pmod{6}$,  $\n_\k \equiv 4l\pmod{6}$, $\k=2,3$, a solution exists if and only if $0\le \n_1, \n_\k\le l$.  Let $\r_1$ and $\r_\k$ $(\k=2,3)$ be the least nonnegative residue of $5l$ and $4l$ modulo 6, respectively. Thus
\[\r_1=
\begin{cases}
5 & \mbox{ if } l\equiv 1\pmod 6,\\
4 & \mbox{ if } l\equiv 2\pmod 6,\\
3 & \mbox{ if } l\equiv 3\pmod 6,\\
2 & \mbox{ if } l\equiv 4\pmod 6,\\
1 & \mbox{ if } l\equiv 5\pmod 6,\\
0 & \mbox{ if } l\equiv 0\pmod 6;
\end{cases}
\quad \r_2=\r_3 =
\begin{cases}
4 & \mbox{ if } l\equiv 1\pmod 6,\\
2 & \mbox{ if } l\equiv 2\pmod 6,\\
0 & \mbox{ if } l\equiv 3\pmod 6,\\
4 & \mbox{ if } l\equiv 4\pmod 6,\\
2 & \mbox{ if } l\equiv 5\pmod 6,\\
0 & \mbox{ if } l\equiv 0\pmod 6.
\end{cases}
\]
Thus, this implies that
\[l=\n_1+\n_2+\n_3\ge \r_1+\r_2+\r_3= \begin{cases}
13 & \mbox{ if } l\equiv 1\pmod 6,\\
8 & \mbox{ if } l\equiv 2\pmod 6,\\
3 & \mbox{ if } l\equiv 3\pmod 6,\\
10 & \mbox{ if } l\equiv 4\pmod 6,\\
5 & \mbox{ if } l\equiv 5\pmod 6,\\
0 & \mbox{ if } l\equiv 0\pmod 6.
\end{cases}\]
Thus, when $l=2, 4, 7$, there is no valid $\n_\k$. Following are all solutions (without loss of generality, we assume $\n_2\ge \n_3$) for  $l\in [3,12]\setminus\{4,7\}$.
\[\begin{array}{c|l}
l & (\n_1, \n_2, \n_3)\\\hline
3 &  (3,0,0)\\
5 &  (1,2,2)\\
6 &  (6,0,0), (0,6,0)\\
\end{array} \qquad \begin{array}{c|l}
l & (\n_1, \n_2, \n_3)\\\hline
8 &  (4,2,2),\\
9 &  (9,0,0), (3, 6, 0)\\
10 & (2,4,4)
\end{array}
 \qquad \begin{array}{c|l}
l & (\n_1, \n_2, \n_3)\\\hline
11 &   (7, 2, 2), (1,8,2)\\
12 & (12, 0,0), (6, 6,0), (0,6,6)
\end{array}
\]
Now $\n_\k=\r_\k +6s_\k$ for some $s_\k\ge 0$, $\k=1,2,3$. So $l=\r_1+\r_2+\r_3+6(s_1+s_2+s_3)$. Suppose $l=6t+\r$ where $t\ge 2$ and $0\le \r\le 5$. Therefore, we may choose
\[\begin{array}{c|*{6}{c|}}
& \r=0 & \r=1 & \r=2 & \r=3 & \r=4 & \r=5\\\hline
(s_1, s_2, s_3)= & (t,0,0) & (0, t-2,0) & (0, t-1,0) & (0,t, 0) & (0, t-1, 0) & (0, t, 0)
\end{array} .\]
Thus, there is always a solution for $(\m_\k, \n_\k)$ if and only if $l\in[3,\infty)\setminus\{4,7\}$. Moreover, not all the components are isomorphic. In general, all the solutions are given by
\[C(2l, (4l)^{[\frac{5l-\n_1}{6}]}, (4l+2)^{[\n_1]}) + H_2+H_3, H_\k\cong C((4l)^{[\frac{10l+6-\n_\k}{6}]}, (4l+2)^{[\n_\k]}), \k=2,3  \] 
where $l\ge 2$, $\n_\k\le l$, $\n_\k = \r_\k + 6s_\k$, $\r_1\equiv 5l\pmod{6}$, $\r_2,\r_3\equiv 4l\pmod{6}$, $\sum^3_{\k=1}\n_\k = l = \sum^3_{\k=1} (\r_\k + 6s_\k)$. Let $l=3t\ge 3$. One possible family of solutions is given by $(\n_1,\n_2,\n_3) = (3t, 0,0)$ and $(\m_1,\m_2,\m_3) = (2t,5t+1,5t+1)$. Therefore, the corresponding family of solutions is $C((12t)^{[2t]}, (12t+2)^{[3t]}) + 2C((12t)^{[5t+1]}), t\ge 1$.

\begin{example} For $l=9$, we get $\r_1=3$, $\r_2=\r_3=0$ with corresponding $(\n_1,\n_2,\n_3) = (9,0,0), (3,6,0)$. Thus, the only two  solutions are $C(18, 36^{[6]}, 38^{[9]}) + 2C(36^{[16]})$ and $C(18, 36^{[7]}, 38^{[3]}) + C(36^{[15]}, 38^{[6]})+C(36^{[16]})$.  \rsq
\end{example}


\ms\nt {\bf Case C.}  In this case, $s=4l+3\ge 7$, $y=(8l+5)(2l+1)$ and $D=[1,8l+4]\cup\{\g_1,\g_2 \mid \g_1+\g_2=y \mbox{ and } \g_1,\g_2 \mbox{  not adjacent to any end-edge labels}\}$. There are $3l+2$ induced paths of length $4l$, namely, $(2l+j,8l+5-j)_{4l}$, $1\le j\le 3l+2$, with end-edge labels sum $10l+5 = 5(2l+1)$, and $l$ induced paths of length $4l+2$, namely, $(i,2l+1-i)_{4l+2}$, $1\le i\le l$, with end-edge labels sum $2l+1$. Moreover, one of the paths must be split into two paths of even length at least 4. Thus, $l\ge 2$. We have two subcases.

\ms\nt {\bf Subcase (I).} Consider the induced paths of length $4l$ with consecutive edge labels  $A_{2l}(2l+j; 8l+4) \diamond A_{2l}(16l^2+8l+1 - j; -8l-4)$, $1\le j\le 3l+2$. Suppose this path is split into two paths of length $2r$ and $4l-2r$ for $2\le r\le 2l-2$, namely, $(2l+j,16l^2+8l+1-j-(r-1)(4(2l+1)))_{2r}$ with end-edge labels sum, say, $\sg_1 = (2l+j)+(16l^2+8l+1-j-(r-1)(4(2l+1)) = 16l^2+10l+1+4(2l+1)-4r(2l+1) = (8l+5-4r)(2l+1)$, and $((2l+j)+4r(2l+1), 8l+5-j)_{4l-2r}$ with end-edge labels sum, say, $\sg_2 = 10l+5+4r(2l+1) = (5+4r)(2l+1)$. Note that $\{(8l+5-4r, 5+4r) \mid 2\le r\le l\} = \{(5+4r,8l+5-4r)\mid l\le r\le 2l-2\}$. By symmetry, we need only consider $2\le r\le l$. Without loss of generality, we may just take $j=1$ to get the paths $(2l+1,(8l+4-4r)(2l+1))_{2r}$ and $((4r+1)(2l+1), 8l+4)_{4l-2r}$. Since the induced paths of length $2r$ and $4l-2r$ must not belong to the same component of each solution, all the solutions are $C(2r, (4l)^{[t_1]}, (4l+2)^{[4r-5t_1]}) + C(4l-2r, (4l)^{[t_2]}, (4l+2)^{[8l-4r-5t_2]}) + C((4l)^{[3l+1-t_1-t_2]}, (4l+2)^{[5t_1+5t_2-7l]})$, where $2\le r\le l$, $4r-l\le 5t_1\le 4r$,  $7l-4r\le 5t_2 \le 8l-4r$, $7l\le 5(t_1+t_2)$, and not all the components are isomorphic.

\ms\nt Note that, if $r=l$, then the solutions are of the form \[C(2r, (4l)^{[t_1]}, (4l+2)^{[s_1]})+C(2r, (4l)^{[t_2]}, (4l+2)^{[s_2]})+C((4l)^{[3l+1-t_1-t_2]}, (4l+2)^{[l-s_1-s_2]}),\] for some suitable $s_1, s_2$. By symmetry, we may assume $t_1\le t_2$.

\begin{example} Let $l=2,3$. One can easily verify that there is no solutions.

\ms\nt Let $l=4$, we have  $2\le r\le 4$, $4r-4\le 5t_1\le 4r$, $28-4r\le 5t_2\le 32-4r$ and $6\le 6(t_1+t_2)$.
If $r=2$, then $t_1 = 1$ and $t_2 =4$, contradicting $t_1+t_2\ge 6$. So there is no solution. If $r=3$, then  $t_1=2$ and $t_2=4$. The solution is $C(6, 16^{[2]}, 18^{[2]}) + C(10, 16^{[4]}) + C(16^{[7]}, 18^{[2]})$. If $r=4$, then  $3\le t_1 \le t_2 \le 3$.  So, $t_1=t_2=3$.  The solution is $C(8, 16^{[3]}, 18) + C(8, 16^{[3]}, 18) + C(16^{[7]}, 18^{[2]})$.

\ms\nt Let $l=5$, we have $2\le r\le 5$,  $4r-5\le 5t_1\le 4r$, $35-4r\le 5t_2\le 40-4r$ and $7\le t_1+t_2$.
If $r=2$, then $t_1=1$ and $t_2=6$. The solution is $C(4,20, 22^{[3]}) + C(16, 20^{[6]}, 22^{[2]}) + C(20^{[9]})$. If $r=3$,  then  $t_1=2$ and $t_2=5$. The solution is $C(6,20^{[2]}, 22^{[2]}) + C(14, 20^{[5]}, 22^{[3]}) + C(20^{[9]})$. If $r=4$, then $t_1=3$ and $t_2=4$. The solution is $C(8,20^{[3]}, 22^{[1]}) + C(12,20^{[4]}, 22^{[4]}) + C(20^{[9]})$. If $r=5$, then  $3\le t_1\le t_2\le 4$ and $t_1+t_2\ge 7$. So, $t_1=3$ and $t_2=4$,  or $t_1=t_2=4$. The solutions are $C(10, 20^{[3]}, 22^{[5]}) + C(10, 20^{[4]}) + C(20^{[9]})$,  and $2 C(10, 20^{[4]}) + C(20^{[8]}, 22^{[5]})$. \rsq
\end{example}


\ms\nt{\bf Subcase (II).} Consider the induced paths of length $4l+2$ with consecutive edge labels $A_{2l+1}(i; 8l+4)\diamond A_{2l+1}(16l^2+10l+1 - i; -8l-4)$, $1\le i\le l$. By the same argument as above, suppose the induced path $(1, 2l+1)_{4l+2}$ is split into two paths to get $(1,16l^2+10l-4(r-1)(2l+1))_{2r}= (1,16l^2+18l+4-4r(2l+1))_{2r}$ with end-edge labels sum $\sg_1=(8l+5-4r)(2l+1)$, and $(1+4r(2l+1), 2l)_{4l+2-2r}$ with end-edge labels sum $\sg_2 = (4r+1)(2l+1)$ for $r\le l$. Thus, all the solutions are
\begin{align*} &\quad C(2r,(4l)^{[t_1]}, (4l+2)^{[4r-5t_1]}) + C(4l+2-2r, (4l)^{[t_2]},(4l+2)^{[8l+4-4r-5t_2]}) \\ & + C((4l)^{[3l+2-t_1-t_2]}, (4l+2)^{[5t_1+5t_2-7l-5]}), \end{align*} where $r\le l$,  $4r-l +1\le 5t_1\le 4r$, $7l+5-4r \le 5t_2 \le 8l+4-4r$, $7l+5\le 5(t_1+t_2)$, and not all the components are isomorphic.

\begin{example}   Let $l=2,3$. One can easily verify that there is no solutions.

\ms\nt Let $l=4$, we have  $2\le r \le 4$ and $t_1+t_2 \ge 7$. If $r=2$, then  $t_1=1$ and $t_2= 5$. There is no solution. If $r=3$, then $t_2$ does not exist. There is no solution. If $r=4$, then $t_1 =3$, $t_2 =4$. The  solution is $C(8,16^{[3]},18) + C(10,16^{[4]}) + C(16^{[7]},18^{[2]})$.

\ms\nt Let $l=5$, we have $2\le r\le 5$ and $t_1+t_2\ge 8$. If $r=2$, then $t_1=1$, $t_2=7$. The solution is $C(4,20,22^{[3]}) + C(18,20^{[7]},22) + C(20^{[9]})$. If $r=3$, then $t_1=2$, $t_2=6$, the solution is $C(6,20^{[2]},22^{[2]}) + C(16,20^{[6]},22^{[2]}) + C(20^{[9]})$. If $r=4$, then $t_1=3$, $t_2=5$. The solution is $C(8,20^{[3]}, 22) +C(14,20^{[5]},22^{[5]}) + C(20^{[9]})$. If $r=5$, then $t_1=t_2=4$. The solution is $C(10,20^{[4]}) + C(12,20^{[4]}, 22^{[4]}) + C(20^{[9]})$. \rsq
\end{example}

\ms\nt{\bf Case D-1.}
In this case, $s=3l\ge 6$, $y=2(3l-1)(2l-1)$ and $D=[1,3(2l-1)]\cup \{y/2\}$. There are $l-1$ induced paths of length $4l-2$ with edge labels $A_{2l-1}(i; 6l-3)\diamond A_{2l-1}((2l-1)(6l-5)-i; -6l+3) = (i,2l-1-i)_{4l-2}$ for $1\le i\le l-1$ with end-edge labels sum $2l-1$; $2l-1$ induced paths of length $4l-4$ with edge labels $A_{2l-2}(2l-2+j; 6l-3)\diamond A_{2l-2}((2l-1)(6l-5)-2l+2-j; -6l+3) = (2l-2+j, 6l-2-j)_{4l-4}$ for $1\le j\le 2l-1$ with end-edge labels sum $8l-4 = 4(2l-1)$; and an induced path of length $2l-2$ with edge labels $A_{l-1}(4l-2; 6l-3)\diamond A_{l-1}((6l-7)(2l-1); -6l+3) = (4l-2,(3l-1)(2l-1))_{2l-2}$ with end-edge labels sum $(3l+1)(2l-1)$. Moreover, one of the paths must be split into two paths of even length at least 4. We have three subcases.

\ms\nt{\bf Subcase (I).} Suppose the induced path of length $2l-2$, namely $(4l-2,(3l-1)(2l-1))_{2l-2}$, is split, then $l\ge 5$. For $r\le (l-1)/2$, split it  into two paths of length $2r$ and $2l-2-2r$, namely, $(4l-2, (6l-3r-4)(2l-1))_{2r}$ with end-edge labels sum $\sg_1=(6l-3r-2)(2l-1)$, and $((3r+2)(2l-1), (3l-1)(2l-1))_{2l-2-2r}$ with end-edge labels sum $\sg_2 = (3l+3r+1)(2l-1)$. Since these two induced paths must not belong to the same component of each solution and every component that contains one of them must contain at least one other induced path, all the solutions are
\begin{align*} &\quad C(2r, (4l-4)^{[t_1]}, (4l-2)^{[3r-4t_1]}) + C(2l-2-2r, (4l-4)^{[t_2]}, (4l-2)^{[3l-3r-3-4t_2]})\\ & + C((4l-4)^{[2l-1-t_1-t_2]}, (4l-2)^{[4t_1+4t_2+2-2l]}),\end{align*}
where $r\le (l-1)/2$, $3r-l+1 \le 4t_1 \le 3r$, $2l-3r-2 \le 4t_2\le 3l-3r-3$, $2l-2 \le 4(t_1+t_2)$, and not all the components are isomorphic. By the same argument as in Case C, if $2r=l-1$, then we may assume $t_1\le t_2$.

\begin{example} Let $l=5$, we have  $r=2=(l-1)/2$. Thus $2\le 4t_1\le 4t_2\le 6$. So $t_1 = t_2=1$. The solution is $2C(4,16,18^{[2]})+C(16^{[7]})$.  Let $l=6$, we have $r=2$, $t_1+t_2 = 3$ with $t_1=1$, $t_2=2$. The solution is $C(4,20,22^{[2]}) + C(6,20^{[2]},22)+C(20^{[8]}, 22^{[2]})$.  \rsq
\end{example}

\ms\nt{\bf Subcase (II).} Suppose one induced path of length $4l-4$ is split, then $l\ge 3$. For $r\le l-1$, by the same argument as in Case C, we split the induced path $(2l-1, 6l-3)_{4l-4}$ to get $(2l-1, (6l-3-3r)(2l-1))_{2r}$ with end-edge labels sum $\sg_1=(6l-2-3r)(2l-1)$ and $((3r+1)(2l-1), 6l-3)_{4l-4-2r}$ with end-edge labels sum $\sg_2=(3r+4)(2l-1)$. Note that the induced path $(4l-2,((3l-1)(2l-1))_{2l-2}$ must belong to exactly one of the components. Since every component that contains $(2l-1, (6l-3-3r)(2l-1))_{2r}$ and $((3r+1)(2l-1), 6l-3)_{4l-4-2r}$ must also contain at least one other induced path, all the solutions are one of the following three types:
\begin{enumerate}[(1)]
\item  Type $S_1: $ $C(2r, 2l-2, (4l-4)^{[t_1]}, (4l-2)^{[3r-4t_1-3l-1]}) + C(4l-4-2r, (4l-4)^{[t_2]}, (4l-2)^{[6l-6-3r-4t_2]}) + C((4l-4)^{[2l-2-t_1-t_2]}, (4l-2)^{[4t_1+4t_2+6-2l]})$, where $r\le l-1$, $3r-4l \le 4t_1 \le 3r-3l-1$, $5l-5-3r \le 4t_2 \le 6l-6-3r$, $2l - 6\le 4(t_1+t_2)$, and not all the components are isomorphic.
\item  Type  $S_2: $ $C(2r, (4l-4)^{[t_1]}, (4l-2)^{[3r-4t_1]}) + C(4l-4-2r, 2l-2, (4l-4)^{[t_2]}, (4l-2)^{[3l-7-3r-4t_2]}) + C((4l-4)^{[2l-2-t_1-t_2]}, (4l-2)^{[4t_1+4t_2+6-2l]})$, where $r\le l-1$, $3r-l+1\le 4t_1\le 3r$, $2l-6-3r \le 4t_2\le 3l-7-3r$, $2l-6 \le 4(t_1+t_2)$, and not all the components are isomorphic.
\item Type $S_3: $ $C(2r, (4l-4)^{[t_1]}, (4l-2)^{[3r-4t_1]}) + C(4l-4-2r, (4l-4)^{[t_2]}, (4l-2)^{[6l-6-3r-4t_2]}) +C(2l-2, (4l-4)^{[2l-2-t_1-t_2]}, (4l-2)^{[4t_1+4t_2+5-5l]})$, where $r\le l-1$, $3r-l+1\le 4t_1\le 3r$, $5l-5-3r  \le 4t_2\le 6l-6-3r$,  $5l-5 \le 4(t_1+t_2)$. By the same argument as in Case C, if $r=l-1$, then we may assume $t_1\le t_2$. Note that the three components are isomorphic if and only if $2r=4l-4-2r=2l-2$ and $t_1=t_2=2l-2-t_1-t_2$, i.e., $r=l-1$ and $t_1=t_2=\frac{2l-2}{3}$. The solutions are $3C(2l-2, (4l-4)^{[(2l-2)/3]}, (4l-2)^{[(l-1)/3]})$, $l \equiv 1\pmod{3}$.
\end{enumerate}

\begin{example} Let $l=3$, we have $r=2$ with $t_1$  does not exist for Types $S_1$ and $S_2$, while Type $S_3$ gives $t_1+t_2=3$, $t_1=t_2=1$, a contradiction.

\ms\nt Let $l=4$, we have $r=2,3$ and both Types $S_1, S_2$ give  $t_1, t_2<0$, respectively, a contradiction. Now, Type $S_3$ gives $t_1+t_2=4,5$.
 \begin{enumerate}[(a)]
\item If $r=2$, then $t_1=1$ and $t_2 = 3$, The solution is $C(4,12, 14^{[2]}) + C(8, 12^{[3]}) + C(6, 12^{[2]}, 14)$.
\item  If $r=3 = l-1$, then $t_1=t_2=2$. The solution is $3C(6, 12^{[2]}, 14)$. \rsq
\end{enumerate}
\end{example}

\ms\nt{\bf Subcase (III).} Suppose one induced path of length $4l-2$ is split, then $l\ge 3$. For $r\le l-1$, by  the same argument as in Case C, we split the induced path $(1,2l-2)_{4l-2}$ to get $(1, 12l^2-10l+1-3r(2l-1))_{2r}$ with end-edge labels sum $\sg_1= (6l-2-3r)(2l-1)$ and $(1+3r(2l-1), 2l-2)_{4l-2-2r}$ with end-edge labels sum $\sg_2 = (3r+1)(2l-1)$. Since every component that contains of $(1, 12l^2-10l+1-3r(2l-1))_{2r}$ or $(1+3r(2l-1), 2l-2)_{4l-2-2r}$ must have at least one other induced path, all the solutions are one of the following three types:
\begin{enumerate}[(1)]
\item Type $S_1: $ $C(2r, 2l-2, (4l-4)^{[t_1]}, (4l-2)^{[3r-3l-1-4t_1]}) + C(4l-2-2r, (4l-4)^{[t_2]}, (4l-2)^{[6l-3-3r-4t_2]}) + C((4l-4)^{[2l-1-t_1-t_2]}, (4l-2)^{[4t_1+4t_2+1-5l]})$, where $r\le l-1$,  $3r-4l+1 \le 4t_1\le 3r-3l-1$,  $5l-1-3r\le 4t_2\le 6l-3-3r$, $5l-1\le 4(t_1+t_2) \le 6l-3$.
\item Type $S_2: $ $C(2r,  (4l-4)^{[t_1]}, (4l-2)^{[3r-4t_1]}) + C(4l-2-2r, 2l-2, (4l-4)^{[t_2]}, (4l-2)^{[3l-4-3r-4t_2]}) + C((4l-4)^{[2l-1-t_1-t_2]}, (4l-2)^{[4t_1+4t_2+1-5l]})$, where $r\le l-1$, $3r-l+2 \le 4t_1\le 3r$,  $2l-2-3r \le 4t_2\le 3l-3r-4$, $5l-1 \le 4(t_1+t_2)$.
\item Type $S_3: $ $C(2r,  (4l-4)^{[t_1]}, (4l-2)^{[3r-4t_1]}) + C(4l-2-2r, (4l-4)^{[t_2]}, (4l-2)^{[6l-3-3r-4t_2]}) + C(2l-2, (4l-4)^{[2l-1-t_1-t_2]}, (4l-2)^{[4t_1+4t_2+1-5l]})$, where $r\le l-1$, $3r-l+2 \le 4t_1\le 3r$,  $5l-1-3r \le 4t_2\le  6l-3-3r$, $5l-1 \le 4(t_1+t_2)$.
\end{enumerate}
\nt Not all the components are isomorphic.

\begin{example}  Let $l=3$, we have $r=2$ with $t_1$ does not exist for Types $S_1, S_2$ and $S_3$. There is no solution. Let $l=4$, we have $r=2,3$,
 \begin{enumerate}[(a)]
\item If $r=2$, then $t_2$ does not exist for Types $S_1$ and $S_3$, while $t_1+t_2$ is not in the given range for Type $S_2$. There is no solution.
\item If $r=3$, then $t_2$ does not exist for Types $S_1$ and $S_2$. Type $S_3$ gives $t_1=2$, $t_2=3$. The solution is $C(6,12^{[2]}, 14) + C(8, 12^{[3]}) + C(6, 12^{[2]}, 14)$.
\end{enumerate}
\end{example}

\ms\nt {\bf Case D-2.}
In this case, $s=3l+1\ge 7$, $y=2l(6l-1)$ and $D=[1,6l-1]\cup \{l(6l-1)\}$. There are $3l-1$ induced paths of length $4l-2$ with edge labels $A_{2l-1}(i; 6l-1) \diamond A_{2l-1}((2l-1)(6l-1)-i; -6l+1) = (i,6l-1-i)_{4l-2}$ for $1\le i\le 3l-1$ with end-edge labels sum $6l-1$, and an induced path of length $2l-2$ with edge labels $A_{l-1}(6l-1; 6l-1)\diamond A_{l-1}((2l-2)(6l-1); -6l+1) = (6l-1, (l-1)(6l-1))_{2l-2}$ with end-edge labels sum $l(6l-1)$. Moreover, one of the paths must be split into two paths of even length at least 4. We have two subcases.

\ms\nt{\bf Subcase (I).} Suppose the induced path of length $2l-2$ is split, then $l\ge 5$. For $r\le (l-1)/2$, split it into two paths of length  $2r$ and $2l-2-2r$, namely $(6l-1, r(6l-1))_{2r}$ with end-edge labels sum $\sg_1 = (r+1)(6l-1)$, and $((2l-r)(6l-1),(l-1)(6l-1))_{2l-2-2r}$ with end-edge labels sum $\sg_2 = (3l-r-1)(6l-1)$. Since these two induced paths must not belong to the same component of each solution, all the possible 3-component solutions set is $\{C(2r, (4l-2)^{[2l-r-1]}) + C((2l-2-2r, (4l-2)^{[r+1-l]}) + C((4l-2)^{[2l-1]})\}$. Since component $C((2l-2-2r, ((4l-2)^{[r+1-l]})$ must have at least one induced cycle of length $4l-2$, we have $r\ge l$, contradicting $r\le (l-1)/2$. Thus, there is no solution by splitting this path.

\ms\nt{\bf Subcase (II).} Suppose one induced path of length $4l-2$ is split, then $l\ge 3$. For $r \le l-1$, by the same argument as in Case C, we split the induced path $(1,6l-2)_{4l-2}$ to get $(1, (2l-1)(6l-1)-1-(r-1)(6l-1))_{2r}$ with end-edge labels sum $\sg_1 = (2l-r)(6l-1)$ and $(1+r(6l-1), 6l-2)_{4l-2-2r}$ with end-edge labels sum $\sg_2 = (r+1)(6l-1)$.  All the solutions must be of type
\begin{align*} & C(2r,(2l-2),(4l-2)^{[r-l]}) + C(4l-2-2r, (4l-2)^{[2l-r-1]}) + C((4l-2)^{[2l-1]}),\\
\mbox{ or } & C(2r, 4l-2)^{[r]} + C(4l-2-2r, 2l-2, (4l-2)^{[l-r-1]}) + C((4l-2)^{[2l-1]}),\\
\mbox{ or } & C(2r, (4l-2)^{[r]}) + C(4l-2-2r, (4l-2)^{[2l-r-1]}) + C(2l-2, (4l-2)^{[l-1]}).\end{align*}
Note that the non-degree 2 vertex of $C((4l-2)^{[2l-1]})$ (and $C(2l-2, (4l-2)^{[l-1]})$) has induced vertex label $(2l-1)(6l-1) < y$, a contradiction. Thus, there is no solution by splitting this path.

\ms\nt{\bf Case E.} In this case, $s=2l+1\ge 5$, $y=(2l-1)(4l-1)$ and $D =[1, 2(2l-1)]\cup\{\gamma_1, \gamma_2, \gamma_3,\gamma_4\}$, where $\g_1+\g_3 = \g_2+\g_4=y$. There are $l$ induced paths of length $4l-4$, namely $A_{2l-2}(2l-2+j;-4l+2)\diamond A_{2l-2}(8l^2-12l+5-j;-4l+2) = (2l-2+j, 4l-1-j)_{4l-4}$ for $1\le j\le l$, with end-edge labels sum $3(2l-1)$, and $l-1$ induced path(s) of length $4l-2$, namely $A_{2l-1}(i,4l-2)\diamond A_{2l-1}(8l^2-10l+3-i;-4l+2) = (i,2l-1-i)_{4l-2}$ for $1\le i\le l-1$, with end-edge labels sum $(2l-1)$. Moreover, two of these paths are split into two paths of even length at least 4, respectively, or else one of these paths is split into three paths of even length at least 4.

\ms\nt{\bf Subcase (I).} Suppose we split two induced paths into two distinct paths respectively. We have the following three subcases.
\begin{enumerate}[(a)]
\item Both the induced paths are of length $4l-4$. For $r_1, r_2\le l-1$, by the same argument as in Case C, split the induced path $(2l-1,4l-2)_{4l-4}$ (respectively, path $(2l,4l-3)_{4l-4}$) to get $Q_1=(2l-1, 2(2l-1)(2l-1-r_1))_{2r_1}$ with end-edge labels sum $\sg_1 = (2l-1)(4l-1-2r_1)$ and $Q_3=((2l-1)(2r_1+1), 4l-2)_{4l-4-2r_1}$ with end-edge labels sum $\sg_2 = (2l-1)(2r_1+3)$ (respectively, to get $Q_2 = (2l, 2(2l-1)(2l-1-r_2)-1)_{2r_2}$ with end-edge labels sum $\sg'_1 =(2l-1)(4l-1-2r_2)$ and $Q_4=((2l-1)(2r_2+1)+1, 4l-3)_{4l-4-2r_2}$ with end-edge labels sum $\sg'_2 = (2l-1)(2r_2+3)$.

    {\bf Claim: } In each of $\{Q_1, Q_3\}$, $\{Q_2, Q_4\}$, $\{Q_1, Q_2\}$, both the induced paths are not in the same component.

    {\bf Proof of the claim: } Since $\sg_1 + \sg_2 = \sg'_1 + \sg'_2 =(2l-1)(4l+2)>y$ and $\sg_1+\sg'_1\ge (2l-1)(4l+2)>y$, we have the claim. \rsq

    If both $Q_1,Q_4$ (and $Q_2,Q_3$) are in the same component, then  $\sg_1+\sg'_2 =(2l-1)(4l+2-2r_1+2r_2) \le (2l-1)(4l-1)$ so that $2r_1\ge2r_2+3$ (and $\sg'_1+\sg_2\le y$ gives $2r_2\ge 2r_1+3$). Therefore, $2r_1\ge 2r_1+6$, a contradiction.
Thus, $Q_i$, $1\le i\le 4$, are in exactly three components. We have the following three subcases.
\begin{enumerate}[(1)]
 \item  If $Q_1$, $Q_2$ are in distinct components respectively and $Q_3,Q_4$ are in the third component, then $\sg_2+\sg'_2 = (2l-1)(2r_1+2r_2+6)\le y$ gives $8\le 2r_1+2r_2\le 4l-7$ so that $l\ge 4$. All the solutions are
\begin{align*}&\quad\ C(2r_1, (4l-4)^{[t_1]}, (4l-2)^{[2r_1-3t_1]}) + C(2r_2,(4l-4)^{[t_2]}, (4l-2)^{[2r_2-3t_2]})\\
    & + C(4l-4-2r_1, 4l-4-2r_2, (4l-4)^{[l-2-t_1-t_2]}, (4l-2)^{[l+3t_1+3t_2-2r_1-2r_2-1]}),\end{align*}
where $r_1, r_2\le l-1$, $2r_1-l+1 \le 3t_1 \le 2r_1$,  $2r_2-l+1 \le 3t_2 \le 2r_2$, $t_1+t_2\le l-2$, $2r_1+2r_2-l+1\le 3(t_1+t_2) \le 2(r_1+r_2)\le 4l-7$, and not all the components are isomorphic. By symmetry, we may assume if $r_1=r_2$, then $t_1\le t_2$, and if $t_1 = t_2$, then $r_1\le r_2$.

\item  If $Q_1,Q_4$ are in distinct components respectively and $Q_2, Q_3$ are in the third component, then $\sg'_1 +\sg_2\le y$ gives $4\le 2r_1\le 2r_2 -3 \le 2l-5$ so that $r_2\ge 4$, $l\ge 5$. All the solutions are
   \begin{align*}&\quad\ C(2r_1, (4l-4)^{[t_1]}, (4l-2)^{[2r_1-3t_1]}) + C(4l-4-2r_2, (4l-4)^{[t_2]}, (4l-2)^{[4l-2r_2-3t_2-4]}) \\ & +
   C(2r_2, 4l-4-2r_1, (4l-4)^{[l-2-t_1-t_2]}, (4l-2)^{[3t_1+3t_2+2r_2+3-2r_1-3l]}), \end{align*}
   where $4\le 2r_1 \le 2r_2-3\le 2l-5$, $2r_1-l+1\le 3t_1 \le 2r_1$, $3l-2r_2-3\le 3t_2\le 4l-2r_2-4$, $2r_1+3l-2r_2-3 \le 3(t_1+t_2) \le 2r_1-2r_2+4l-4 \le 4l-7$, and not all the components are isomorphic. By  symmetry, we may assume if $r_1+r_2 = 2l-2$, then $t_1\le t_2$, and the first two components are isomorphic when equality holds.

\item  If $Q_2,Q_3$ are in distinct components respectively and $Q_1,Q_4$ are in the third component, then we may swap the notation $r_1,r_2$ so that $Q_1,Q_3$, and $Q_2,Q_4$ are swapped. Therefore, the solutions are isomorphic to those in Subcase (2) above.
  \end{enumerate}
\end{enumerate}

\begin{example} Consider (1) above.   For $l=4$, we get $r_1=r_2=2$, $t_1=t_2=1$. The solution is $2C(4,12, 14)+C(8, 8, 14)$.

\ms\nt Consider (2) above.  For $l=5$, we get $r_2=4$, $r_1=2$, $t_1 = 0$, $t_2 = 1$. The solution is $C(8, 12, 18) + C(4, 16, 18) + C(8, 16^{[2]}, 18^{[2]})$.

\ms\nt For $l=6$, we get $4\le 2r_1 \le 2r_2-3\le 7$ so that $2\le r_1\le 3$, $4\le r_2\le 5$ and $t_1+t_2\le 4$.
\begin{enumerate}[(i)]
\item Take $(r_1,r_2) = (2,4)$, we get $t_1=0$, $t_2=0,1$. So, $(r_1,r_2,t_1,t_2) = (2,4,0,0), (2,4,0,1)$. The solutions are $C(8,16,22) + C(4,22^{[4]}) + C(12,20^{[4]})$ and $C(8,16,22) + C(4,20,22) + C(12,20^{[3]},22^{[3]})$, respectively.
\item Take $(r_1,r_2) = (2,5)$, we get $t_1,t_2 = 0,1$ and $t_1+t_2=1, 2$.  So, $(r_1,r_2,t_1,t_2) = (2,5,0,1)$, $(2,5,1,0)$, $(2,5,1,1)$. The solutions are $C(10,16,22^{[3]}) + C(4,20,22) + C(10,20^{[3]},22)$,  and $C(10,16,20) + C(4,20,22) + C(10,20^{[2]},22^{[4]})$, respectively.
\item Take $(r_1,r_2) = (3,5)$, we get $t_1 = 0$, $t_2=1,2$. So, $(r_1,r_2,t_1,t_2) = (3,5,0,1)$, $(3,5,0,2)$. The solutions are $C(10,14,22) + C(6,20,22^{[3]}) + C(10,20^{[3]},22)$ and $C(10,14,22) + C(6,20^{[2]}) + C(10,20^{[2]},22^{[4]})$, respectively. \rsq
\end{enumerate}
\end{example}

\begin{enumerate}[(a)]
\item[(b)] One induced path is of length $4l-4$ and another is of length $4l-2$. In this case, $l\ge 3$. By the same argument as in Subcase (a) above, for $r_1 \le l-1$, we split the induced path $(2l-1,4l-2)_{4l-4}$ to get $Q_1=(2l-1, 2(2l-1)(2l-1-r_1))_{2r_1}$ with end-edge labels sum $\sg_1 = (2l-1)(4l-1-2r_1)$ and $Q_3=((2l-1)(2r_1+1), 4l-2)_{4l-4-2r_1}$ with end-edge labels sum $\sg_2 = (2l-1)(2r_1+3)$. For $r_2 \le l-1$, we split the induced path $(1,2l-2)_{4l-2}$ to get $Q_2 = (1,2l(4l-3)-r_2(4l-2))_{2r_2}$ with end-edge labels sum $\sg'_1 = (2l-1)(4l-1-2r_2)$ and $Q_4 = ((2l-1)(4l-1)-2l(4l-3)+r_2(4l-2), 2l-2)_{4l-2-2r_2}$ with end-edge labels sum $\sg'_2 = (2l-1)(2r_2+1)$. Moreover,   $Q_i$, $1\le i\le 4$, are in exactly three components.

\begin{enumerate}[(1)]
\item If $Q_1,Q_2$ are in distinct components respectively and $Q_3,Q_4$ are in the third component, then $\sg_2 +\sg'_2\le y$ gives $8\le 2r_1+2r_2\le 4l-5$ so that $l\ge 4$. All the solutions are
    \begin{align*} &\quad \ C(2r_1,(4l-4)^{[t_1]}, (4l-2)^{[2r_1-3t_1]})+C(2r_2,(4l-4)^{[t_2]}, (4l-2)^{[2r_2-3t_2]})\\
    & +C(4l-4-2r_1, 4l-2-2r_2, (4l-4)^{[l-1-t_1-t_2]}, (4l-2)^{[l+3t_1+3t_2-2r_1-2r_2-2]}), \end{align*}
    where $2\le r_1, r_2\le l-1$, $2r_1-l+2\le 3t_1\le 2r_1$, $2r_2-l+2\le 3t_2\le 2r_2$, $t_1+t_2\le l-1$,  $2r_1+2r_2+2-l \le 3(t_1+t_2) \le 2r_1+2r_2\le 4l-5$, and not all the components are isomorphic.


\item If $Q_1,Q_4$ are in distinct components respectively and $Q_2,Q_3$ are in the third component, then $\sg'_1+\sg_2 \le y$ gives $7\le 2r_1+3\le 2r_2 \le 2l-2$ so that $r_2\ge 4$, $l\ge 5$. All the solutions are
\begin{align*} & \quad \ C(2r_1,(4l-4)^{[t_1]}, (4l-2)^{[2r_1-3t_1]})+C(4l-2-2r_2,(4l-4)^{[t_2]}, (4l-2)^{[4l-2r_2-3t_2-2]})\\
    & +C(2r_2,4l-4-2r_1,(4l-4)^{[l-1-t_1-t_2]}, (4l-2)^{[3t_1+3t_2+2r_2-2r_1-3l]}),\end{align*}
    where $7\le 2r_1+3\le 2r_2 \le 2l-2$, $2r_1-l+2\le 3t_1\le 2r_1$, $3l-2r_2 \le 3t_2 \le 4l - 2r_2-2$, $t_1+t_2\le l-1$, $3l+2r_1-2r_2\le 3(t_1+t_2) \le 4l+2r_1-2r_2-2$,  and not all the components are isomorphic.

\item If $Q_2,Q_3$ are in distinct components respectively and $Q_1,Q_4$  are in the third component, then $\sg_1+\sg'_2\le y$ gives $5\le 2r_2+1\le 2r_1\le 2l-2$ so that $r_1\ge 3, l\ge 4$. All the solutions are
    \begin{align*}& \quad\ C(2r_2,(4l-4)^{[t_1]}, (4l-2)^{[2r_2-3t_1]})+C(4l-4-2r_1,(4l-4)^{[t_2]}, (4l-2)^{[4l-2r_1-3t_2-4]})\\
    &+C(2r_1, 4l-2-2r_2, (4l-4)^{[l-1-t_1-t_2]}, (4l-2)^{[3t_1+3t_2+2r_1+2-2r_2-3l]}),\end{align*}
    where $5\le 2r_2+1\le 2r_1\le 2l-2$, $2r_2 -l+2\le 3t_1 \le 2r_2$, $3l-2r_1-2 \le 3t_2 \le 4l-2r_1-4$, $t_1+t_2 \le l-1$, $3l+2r_2-2r_1-2\le 3(t_1+t_2) \le 4l+2r_2-2r_1-4$, and not all the components are isomorphic.

     \end{enumerate}
 \end{enumerate}

\begin{example} Consider (1) above. For $l=4$, we get $4\le r_1+r_2\le 5$. Take $r_1=r_2=2$, we get $t_1 = t_2 = 1$, and the solution is $2C(4, 12, 14) + C(8, 10, 12)$.  Take $r_1=2$, $r_2=3$, we get $t_1=1$, $t_2=2$, and the solution is $C(4,12, 14) + C(6, 12^{[2]}) + C(8, 8, 14)$.  Take $r_1=3,r_2=2$, we get $t_1 =2$, $t_2=1$, and the solution is $C(6,12^{[2]}) + C(4, 12, 14) + C(6, 10, 14)$.

\ms\nt Consider (2) above.  For $l=5$, we get $r_1 =2$, $r_2=4$, $t_1=1$, $t_2=3$. The solution is $C(4,16,18) + C(10,16^{[3]},18) + C(8,12,18)$.

\ms\nt Consider (3) above. For $l=4$, we get $r_1=3, r_2=2$, $t_1=1, t_2 = 1,2$ and $t_1+t_2\ge 3$. So, $(r_1,r_2,t_1,t_2) = (3,2,1,2)$. The solution is  $C(4,12,14) + C(6,12^{[2]}) + C(6,10,14)$.
\rsq
\end{example}

\begin{enumerate}[(a)]
\item[(c)] Both the induced paths are of length $4l-2$. By the same argument as in Subcase (a) above, for $r_1, r_2 \le l-1$, we split $(1,2l-2)$ (respectively $(2, 2l-3)$) to get $Q_1 =  (1,2l(4l-3)-r_1(4l-2))_{2r_1}$ with end-edge labels sum $\sg_1 = (2l-1)(4l-1-2r_1)$ and $Q_3 = ((2l-1)(4l-1)-2l(4l-3)+r_1(4l-2), 2l-2)_{4l-2-2r_1}$ with end-edge labels sum $\sg_2 = (2l-1)(2r_1+1)$ (respectively, to get $Q_2 = (2, 2l(4l-3)-1-r_2(4l-2))_{2r_2}$  with end-edge label $\sg'_1 = (2l-1)(4l-1-2r_2)$ and $Q_4 = ((2l-1)(4l-1)+1-2l(4l-3)+r_2(4l-2), 2l-3)_{4l-2-2r_2}$ with end-edge labels sum $\sg'_2 = (2l-1)(2r_2+1)$). Moreover, $Q_i$, $1\le i\le 4$, are in exactly three components.

  \begin{enumerate}[(1)]
  \item If $Q_1,Q_2$ are in distinct components respectively and $Q_3,Q_4$ are in the third component, then $\sg_2 +\sg'_2\le y$ gives $8\le 2r_1+2r_2\le 4l-3$ so that $l\ge 3$. All the solutions are
       \begin{align*} & \quad\ C(2r_1,(4l-4)^{[t_1]}, (4l-2)^{[2r_1-3t_1]})+C(2r_2,(4l-4)^{[t_2]}, (4l-2)^{[2r_2-3t_2]})\\& +C(4l-2-2r_1,4l-2-2r_2, (4l-4)^{[l-t_1-t_2]}, (4l-2)^{[3t_1+3t_2+l-2r_1-2r_2-3]}),\end{align*}
       where $r_1, r_2\le l-1$, $t_1+t_2\le l$, $2r_1-l+3 \le 3t_1\le 2r_1$,  $2r_2 - l + 3\le 3t_2 \le 2r_2$,  $2r_1+2r_2-l+3\le 3(t_1+t_2) \le 2r_1+2r_2 \le 4l-3$, and not all the components are isomorphic. By symmetry, we may assume if $r_1=r_2$, then $t_1\le t_2$, and if $t_1=t_2$, then  $r_1\le r_2$.

  \item If $Q_1,Q_4$ are in distinct components respectively and $Q_2,Q_3$ are in the third component, then $\sg'_1+\sg_2 \le y$ gives $5\le 2r_1+1 \le 2r_2 \le 2l-2$ so that $l\ge 4$ and $r_2\ge 3$.  All the solutions are
       \begin{align*}&\quad\ C(2r_1,(4l-4)^{[t_1]}, (4l-2)^{[2r_1-3t_1]})+C(4l-2-2r_2,(4l-4)^{[t_2]}, (4l-2)^{[4l-2r_2-3t_2-2]})\\ & +C(2r_2,4l-2-2r_1,(4l-4)^{[l-t_1-t_2]}, (4l-2)^{[3t_1+3t_2+2r_2-2r_1-3l-1]}),\end{align*}
       where $5\le 2r_1+1 \le 2r_2 \le 2l-2$, $t_1+t_2\le l$, $2r_1-l+3\le 3t_1 \le 2r_1$, $3l+1-2r_2 \le 3t_2 \le 4l-2r_2-2$,  $2r_1+3l+1-2r_2\le 3(t_1+t_2) \le 2r_1+4l-2r_2-2 \le 4l-3$, and not all the components are isomorphic.
 \item Suppose $Q_2,Q_3$ are in distinct components respectively and $Q_1,Q_4$  are in the third component. Similar to Case~(a)(3), the solutions are isomorphic to those in Case (c)(2) above.
  \end{enumerate}
\end{enumerate}


\begin{example}  Consider (1) above.  For $l=3$, we get $r_1 = r_2=2$ so that $3t_1=3t_2=4$, a contradiction.
 For $l=4$, we get  $2\le r_1\le r_2\le 3$.
\begin{enumerate}[(i)]
\item Take $(r_1,r_2)=(2,2)$, we get $t_1=t_2=1$ with $7\le 3(t_1+t_2) = 6$, a contradiction.
\item Take $(r_1,r_2)=(2,3)$, we get $t_1=1$, $t_2=2$. The solution is $C(4,12,14) + C(6,12^{[2]}) + C(8,10,12)$.
\item Take $(r_1,r_2)=(3,3)$, we get $t_1=t_2=2$. The solution is $2C(6,12^{[2]}) + C(8,8,14)$.
\end{enumerate}

\ms\nt Consider (2) above. For $l=4$, we get $r_1=2$, $r_2=3$ and $7\le 3t_2\le 8$, a contradiction.  For $l=5$, we get $2\le r_1\le 3$, $3\le r_2 \le 4$.
\begin{enumerate}[(i)]
\item Take $(r_1,r_2)=(2,3)$, we get $t_1=1$, $t_2=4$, The solution is $C(4, 16, 18) +C(12, 16^{[4]}) + C(6, 14, 18)$.
\item Take $(r_1,r_2)=(2,4)$, we get $t_1=1$, $t_2=3$. The solution is $C(4,16, 18) +C(10, 16^{[3]}, 18) + C(8, 14, 16)$.
\item Take $(r_1,r_2)=(3,4)$, we get $t_1=2$, $t_2=3$. The solution is $C(6, 16^{[2]}) +C(10, 16^{[3]}, 18) + C(8, 12, 18)$.  \rsq
\end{enumerate}
\end{example}

\ms\nt{\bf Subcase (II).} Suppose we split an induced path into three distinct paths. We have the following two subcases.
\begin{enumerate}[(a)]
\item The induced path is of length $4l-4$ and the 3 distinct paths are $Q_1, Q, Q_2$ of length $2r_1, 2r, 2r_2$ respectively satisfying $r_1,r_2, r\le 2l-6$ and $r_1+r_2+r = 2l-2$. Without loss of generality, let $Q_1 = (2l-1, 2(2l-1)(2l-1-r_1)_{2r_1}$ with end-edge labels sum $\sg_1 = (2l-1)(4l-1-2r_1)=(2l-1)(2r_2+2r+3)$, $Q_2 = ((2l-1)(2r_1+2r+1),4l-2)_{2r_2}$ with end-edge labels sum $\sg_2 = (2l-1)(2r_1+2r+3)$, and $Q=((2l-1)(2r_1+1), (2l-1)(4l-2-2r_1-2r))_{2r}$ with end-edge labels sum $\sg=(2l-1)(4l-1-2r) =(2l-1)(2r_1+2r_2+3)$.
    Clearly $\sg_1+\sg$, $\sg_2+\sg$ and $\sg_1+\sg_2$ are greater than $y$. So $Q_1$, $Q_2$ and $Q$ are in distinct components. Note that $l\ge 4$. By symmetry, we may assume $r_1\le r_2$.
Thus, all the solutions are
\begin{align*}& \quad\ C(2r_1, (4l-4)^{[t_1]}, (4l-2)^{[2r_1-3t_1]}) + C(2r_2, (4l-4)^{[t_2]}, (4l-2)^{[2r_2-3t_2]}) \\&+ C(4l-4-2r_1-2r_2, (4l-4)^{[l-1-t_1-t_2]}, (4l-2)^{[3t_1+3t_2+l -2r_1-2r_2-1]}),\end{align*}
 where $r_1\le r_2$, $r_1+r_2\le 2l-4$, $2r_1-l+1\le 3t_1\le 2r_1$, $2r_2-l+1 \le 3t_2 \le 2r_2$, $2r_1+2r_2-l+1\le 3(t_1+t_2) \le 3(l-1)$. By symmetry, we may assume if $r_1=r_2$, then $t_1\le t_2$, and if $t_1=t_2$, then $r_1\le r_2$.  Moreover, all the three components are isomorphic if and only if $r_1=r_2=r=\frac{2l-2}{3} = 2t$, $2t_1+t_2=3t = 2t_2+t_1$, i.e., $t_1=t_2=t$ for $l=3t+1\ge 4$ if and only if the solutions are $3C(4t, (12t)^{[t]}, (12t+2)^{[t]})$, $t\ge 1$.
\end{enumerate}

\begin{example}
For $l=4$, we get $r_1=r_2=2$, $t_1=t_2=1$. The solution is $3C(4, 12, 14)$.

\ms\nt For $l=5$, we get $2\le r_1\le  r_2$, $r_1+r_2\le 6$, $t_1+t_2 \le 4$.
\begin{enumerate}[(i)]
\item Take $(r_1,r_2) = (2,2)$, one can verify that $t_1=t_2=1$. The solution is $2C(4,16,18)+C(8,16^{[2]},18^{[2]})$.
\item Take $(r_1,r_2)=(2,3)$, one can verify that $(t_1,t_2) \in \{(0,2), (1,1), (1,2)\}$. The solutions are $C(4,18^{[4]}) + 2C(6,16^{[2]})$, $C(4,16,18)+C(6,16,18^{[3]})+ C(6, 16^{[2]})$ and $C(4,16,18) + C(6,16^{[2]}) + C(6,16,18^{[3]})$, respectively.
\item Take $(r_1,r_2)=(2,4)$,  one can verify that $t_1=1,t_2=2$. The solution is $C(4,16,18) + C(8,16^{[2]},18^{[2]}) + C(4,16,18)$.
\item Take $(r_1,r_2)=(3,3)$, one can verify that $(t_1,t_2) \in \{(1,2), (2,2)\}$. The solutions are $C(6,16,18^{[3]}) + C(6,16^{[2]}) + C(4,16,18)$ and $2C(6,16^{[2]}) + C(6,18^{[4]})$, respectively. \rsq
\end{enumerate}
\end{example}

\begin{enumerate}[(a)]
\item[(b)] The induced path is of length $4l-2$ and the 3 distinct paths are $Q_1, Q, Q_2$ of length $2r_1, 2r, 2r_2$ respectively satisfying $r_1,r_2,r\le 2l-5$ and $r_1+r_2+r=2l-1$. Note that $l\ge 4$. Similar to Subcase (a) above, let $Q_1 = (1, 2l(4l-3)-r_1(4l-2))_{2r_1}$ with end-edge labels sum $\sg_1 = (2l-1)(4l-1-2r_1) = (2l-1)(2r_2+2r+1)$, $Q_2 = ((2l-1)(4l-1)-2l(4l-3)+(r_1+r)(4l-2), 2l-2)_{2r_2}$ with end-edge labels sum $\sg_2 = (2l-1)(2r_1+2r+1)$, and $Q=((2l-1)(4l-1)-2l(4l-3)+r_1(4l-2), 2l(4l-3)-(r_1+r)(4l-2))_{2r_1+2r}$ with end-edge labels sum $\sg=(2l-1)(4l-1-2r) = (2l-1)(2r_1+2r_2+1)$. Moreover,  $Q_1$, $Q_2$ and $Q$ are in distinct components.
  Thus, all the solutions are
\[C(2r_1, (4l-4)^{[t_1]} , (4l-2)^{[2r_1-3t_1]}) + C(2r_2, (4l-4)^{[t_2]}, (4l-2)^{[2r_2-3t_2]}) + C(2r, (4l-4)^{[t]}, (4l-2)^{[2r-3t]}),\] where  $t_1+t_2+t=l$, $r_1+r_2+r=2l-1$, $2r'-l+2\le 3t' \le 2r'$ for $(r',t') \in\{(r_1,t_1), (r_2,t_2), (r,t)\}$, $2r'+2r''-l+2\le 3(t'+t'') \le 3l$ for $(r', r'', t', t'') \in \{(r_1,r_2, t_1,t_2), (r_1,r,t_1,t), (r_2,r,t_2,t)\}$.  By symmetry, we may assume $r_1\le r_2\le r$. If equalities hold, we may assume $t_1\le t_2 \le t$. 
 Moreover, if all the three components are isomorphic, then $r_1=r_2=r=\frac{2l-1}{3}$ so that $l\equiv 2\pmod{3}$, and that $t_1=t_2=t=\frac{l}{3}$, i.e., $l\equiv 0\pmod{3}$, a contradiction.
\end{enumerate}

\begin{example} For $l=4$, we get $(r_1, r_2, r)=(2,2,3)$ and $t_1=t_2=1,t=2$. The solution is $2C(4, 12, 14)+C(6, 12^{[2]})$.

\ms\nt For $l=5$, we get $(r_1,r_2,r) = (2,2,5), (2,3,4), (3,3,3)$. Thus, $(t_1,t_2,t)  = (1,1,3), (1,2,2), (1,2,2)$, respectively. The solutions are $2C(4, 16, 18) + C(10, 16^{[3]}, 18)$, $C(4, 16, 18) + C(6, 16^{[2]}) + C(8, 16^{[2]}, 18^{[2]})$, $C(6, 16, 18^{[3]}) + C(6, 16^{[2]}) + C(6, 16^{[2]})$, respectively. \rsq
\end{example}

\ms\nt{\bf Case F.} In this case, $s\ge 5$, $y=2(s-2)(2s-5)$ and $D=[1, 2s-5]\cup\{y/2, \gamma_1, \gamma_2, \gamma_3,\gamma_4\}$, where $\gamma_1+\gamma_3=\gamma_2+\gamma_4=y$. There are $s-3$ induced paths of length $4s-10$, namely $A_{2s-5}(i;2s-5)\diamond A_{2s-5}((2s-5)^2-i; -2s+5)$ for $1\le i\le s-3$, with end-edge labels sum $2s-5$, and an induced path  of length $2s-6$ with edge labels $A_{s-3}(2s-5; 2s-5)\diamond A_{s-3}((2s-6)(2s-5); -2s+5)$ with end-edge labels sum $(s-1)(2s-5)$. Moreover, two of these paths are split into two paths of even length at least 4, respectively, or else one of these paths is split into three paths of even length at least 4. We have two subcases.

\ms\nt{\bf Subcase (I).} Suppose we split two induced paths into two distinct paths respectively. We have the following two subcases.

\begin{enumerate}[(a)]
\item Both the induced paths are of length $4s-10$. For $r_1,r_2\le s-3$, by the same argument as in Case E, split  the induced path $(1,2s-6)$ (respectively, path $(2,2s-7)$) to get $Q_1 = (1, (2s-5)(2s-4-r_1) - 1)_{2r_1}$ with end-edge labels sum $\sg_1 = (2s-5)(2s-4-r_1)$ and $Q_3 = (1+r_1(2s-5), 2s-6)_{4s-10-2r_1}$ with end-edge labels sum $\sg_2 = (2s-5)(r_1+1)$ (respectively, to get $Q_2 = (2, (2s-5)(2s-4-r_2)-2)_{2r_2}$ with end-edge labels sum $\sg'_1 = (2s-5)(2s-4-r_2)$ and $Q_4 = (2+r_2(2s-5),2s-7)_{4s-10-2r_2}$ with end-edge labels sum $\sg'_2 = (2s-5)(r_2+1)$). Note that both $Q_1,Q_3$ (and $Q_2,Q_4$) must not belong to the same component of each solution. By a similar argument as in Case~E(I)(a) and the claim, we have  $Q_i$, $1\le i\le 4$, are in exactly three components, and that $Q_1,Q_2$ are not in the same component (and $Q_3,Q_4$ are in distinct component respectively) of any solutions.

If the induced paths of length $2s-6$ and $Q_i$ are in the same component, $i=1,2$, then the solution graph contains a component $C(2r_i, (4s-10)^{[(r_i-s+1)]}, 2s-6)$. This implies that $r_i-s+1\ge 0$ which contradicts $r_i\le s-3$.
\begin{enumerate}[(1)]
\item  Suppose $Q_1,Q_2$ are in distinct components respectively and $Q_3,Q_4$ are in the third component. Thus, the induced paths of length $2s-6$ and $Q_3, Q_4$ are in the same  component. Therefore, all the solutions are
  \begin{align*} &\quad\ C(2r_1, (4s-10)^{[r_1]}) + C(2r_2, (4s-10)^{[r_2]}) \\&+ C(4s-10-2r_1, 4s-10-2r_2, (4s-10)^{[s-r_1-r_2-5]}), 2s-6) ,\end{align*} where $r_1, r_2\le s-3$, $4\le r_1+r_2\le s-5$, and not all the components are isomorphic. Clearly, $s\ge 9$.

\item Suppose $Q_1,Q_4$ are in distinct components respectively and $Q_2,Q_3$ are in the third component. Thus, the induced paths of length $2s-6$ and $Q_4$ are in the same  component. Therefore, all the solutions are
\begin{align*}& \quad\ C(2r_1, (4s-10)^{[r_1]}) + C(4s-10-2r_2, (4s-10)^{[s-4-r_2]}, 2s-6)\\& + C(2r_2, 4s-10-2r_1, (4s-10)^{[r_2-r_1-1]}),\end{align*}
         where $r_1\le s-3$, $r_2\le s-4$, $1\le r_2-r_1 \le s-4$, and not all the components are isomorphic. Clearly, $s\ge 6$.

\item Suppose $Q_2,Q_3$ are in distinct components respectively and $Q_1,Q_4$ are in the third component. We may swap the notation $r_1,r_2$ so that $Q_1,Q_3$, and $Q_2,Q_4$ are swapped. Therefore, the solutions are isomorphic to those in Case (2) above.
 \end{enumerate}
  \end{enumerate}
\begin{example}  Consider (1) above.  For $s=9$, we get $r_1=r_2=2$. The unique solution is $2C(4,26^{[2]}) + C(12, 22,22)$.

\ms\nt  Consider (2) above. Note that, there is no solution for $s=6$. For $s=7$, we get $r_1=2$ and $r_2=3$. The unique solution is $C(4, 18^{[2]}) + C(8, 12) + C(6, 14)$. \rsq
\end{example}
\begin{enumerate}[(a)]
\item[(b)] One induced path is of length $4s-10$ and another is of length $2s-6$, $s\ge 7$. By the same argument as in Case (a) above, for $r_1\le s-3$, we split the induced path $(1,2s-6)$  to get $Q_1 = (1, (2s-5)(2s-4-r_1) - 1)_{2r_1}$ with end-edge labels sum $\sg_1 = (2s-5)(2s-4-r_1)$ and $Q_3 = (1+r_1(2s-5), 2s-6)_{4s-10-2r_1}$ with end-edge labels sum $\sg_2 = (2s-5)(r_1+1)$. For $r_2 \le (s-4)/2$, we split the induced path $(2s-5, (s-2)(2s-5))$ to get $Q_2 = (2s-5, (2s-5)(2s-5-r_2))_{2r_2}$ with end-edge labels sum $\sg'_1 = (2s-5)(2s-4-r_2)$ and $Q_4 = ((2s-5)(r_2+1), (s-2)(2s-5))_{2s-6-2r_2}$ with end-edge labels sum $\sg'_2=(2s-5)(r_2+s-1)$.
Since $r_1\le s-3$ and $r_2\le \frac{s}{2}-2$, $\sg_1+\sg_2=(2s-5)(2s-3)$, $\sg_1'+\sg_2'=(2s-5)(3s-5)$, $\sg_1+\sg_1'=(2s-5)(4s-8-(r_1+r_2))$ and $\sg_1+\sg'_2 = (2s-5)(2s-4+r_2+s-1-r_1)\ge (2s-5)(3s-5+2-(s-3))-(2s-5)(2s)$ are greater than $y$. So, in each of $\{Q_1, Q_3\}$, $\{Q_2, Q_4\}$, $\{Q_1, Q_2\}$, $\{Q_1, Q_4\}$ both the induced paths are not in the same component. Consequently, $Q_i$, $1\le i\le 4$, are in exactly three components and $Q_1$ is in a component alone.

\begin{enumerate}[(1)]
\item If $Q_1, Q_2$ are in distinct component respectively and $Q_3,Q_4$ are in the third component, then $\sg_2+\sg'_2\le y$ gives $r_1+r_2\le s-4$.  All the solutions are
\[C(2r_1, (4s-10)^{[r_1]}) + C(2r_2, (4s-10)^{[r_2]}) + C(4s-10-2r_1, 2s-6-2r_2, (4s-10)^{[s-4-r_1-r_2]}),\]
 where $r_1\le s-3$, $r_2\le (s-4)/2$, $4\le r_1+r_2\le s-4$, and not all the components are isomorphic. Clearly, $s\ge 8$. By symmetry, we may assume $r_1\le r_2$.

\item  If $Q_1, Q_4$ are in distinct component respectively and $Q_2,Q_3$ are in the third component, then $\sg'_1+\sg_2\le y$ gives $r_1\le r_2-1\le s/2-3$ so that $s\ge 10$ and $r_2\ge 3$. All the solutions are
\[C(2r_1, (4s-10)^{[r_1]}) + C(2s-6-2r_2, (4s-10)^{[s-r_2-3]}) + C(2r_2, 4s-10-2r_1, (4s-10)^{[r_2-r_1-1]}),\]
where $r_1\le r_2-1\le s/2-3$, $r_2-r_1\le s-3$, and not all the components are isomorphic.
 \end{enumerate}
\end{enumerate}

\begin{example}  Consider (1) above. For $s=8$, we get $r_1=r_2=2$. The solution is $2C(4,22^{[2]}) + C(18, 6)$.

\ms\nt Consider (2) above. For $s=10$, we get $r_1=2$, $r_2=3$. The solution is $C(4, 30^{[2]}) + C(8, 30^{[4]}) + C(6,26)$.   \rsq
\end{example}

\ms\nt{\bf Subcase (II).} Suppose we split an induced path into three distinct paths. We have the following two subcases.
\begin{enumerate}[(a)]
\item The induced path is of length $4s-10$, $s\ge 6$ and the 3 distinct paths are $Q_1,Q,Q_2$ of length $2r_1, 2r, 2r_2$ respectively satisfying $r_1, r_2, r\le 2s-9$ and $r_1+r_2+r=2s-5$. Without loss of generality, let $Q_1 = (1,(2s-5)(2s-4-r_1)-1)_{2r_1}$ with end-edge labels sum $\sg_1 = (2s-5)(2s-4-r_1)$, $Q_2 = ((2s-5)(r_1+r)+1, 2s-6)_{2r_2}$ with end-edge labels sum $\sg_2 = (2s-5)(r_1+r+1)=(2s-5)(2s-4-r_2)$, and $Q= (1+r_1(2s-5), (2s-5)(2s-4-r_1-r)-1)_{2r}$ with end-edge labels sum $\sg = (2s-5)(2s-4-r)$. Clearly, each of $Q_1,Q_2,Q$ must belong to distinct component of each solution. By symmetry on $\sg_1, \sg_2, \sg$, we may assume $Q_1$ and the induced paths of length $2s-6$ are in the first component. Thus, all the solutions are
  \[C(2r_1, (4s-10)^{[r_1-s+1]}, (2s-6)) + C(2r_2, (4s-10)^{[r_2]}) + C(2r,  (4s-10)^{[r]}),\]
   where $s-1 \le r_1\le 2s-9$, $2\le r_2\le s-6$, $r_1+r_2+r=2s-5$, $s+1\le r_1+r_2 \le 2s-7$. Note that $s\ge 8$, and not all the components are isomorphic.
\end{enumerate}
\begin{example}
   For $s=8$, we get $r_1=7$, $r_2 = 2$, $r_1+r_2+r=11$ so that $r=2$. The solution is $C(14,10) + C(4, 22^{[2]}) + C(4,22^{[2]})$. \rsq
\end{example}

\begin{enumerate}[(a)]
\item[(b)] The induced path is of length $2s-6$, $s\ge 9$, and the 3 distinct paths are $Q_1,Q,Q_2$ of length $2r_1, 2r, 2r_2$ respectively satisfying $r_1, r_2, r\le s-7$ and $r_1+r_2+r=s-3$. We may assume $Q_1 = (2s-5, (2s-5)(2s-5-r_1))_{2r_1}$ with end-edge labels sum $\sg_1 = (2s-5)(2s-4-r_1)$, $Q_2 = ((2s-5)(r_1+r+1), (s-2)(2s-5))_{2r_2}$ with end-edge labels sum $\sg_2 = (2s-5)(r_1+r+s-1) =(2s-5)(2s-4-r_2)$, and $Q = ((2s-5)(r_1+1), (2s-5)(2s-5-r_1-r))_{2r}$ with end-edge labels sum $\sg = (2s-5)(2s-4-r)$. Clearly, each of $Q_1,Q_2,Q$ must belong to distinct component of each solution. Thus, all the solutions are
     \[C(2r_1, (4s-10)^{[r_1]}) + C(2r_2, (4s-10)^{[r_2]}) + C(2r, (4s-10)^{[r]}),\]
    where $r_1, r_2, r\le s-7$, $r_1+r_2+r=s-3$. By symmetry, we may assume $r_1\le r_2\le r$.

 Note that all components are isomorphic only if $s=3r+3\ge 9$ and $r_1=r=r_2\ge 2$ and the solutions are $3C(2r, (12r+2)^{[r]})$. For $s=3r+3$, we may get other solutions by taking $r_1=s-1$, $r=s$, $r_2=s+1$.
\end{enumerate}

\begin{example}
For $s=9$, we get $r_1=r_2=r=2$. The solution is $3C(4, 26^{[2]})$.

\nt For $s=10$, we get $r_1=r_2=2$, $r=3$. The solution is $2C(4, 30^{[2]}) + C(6, 30^{[3]})$.

\nt For $s=12$, we must get  $(r_1,r_2,r)=(2,3,4)$, $(3,3,3)$. The solutions are  $C(4, 38^{[2]})+C(6, 38^{[3]})+C(8, 38^{[4]})$ and $3C(6, 38^{[3]})$, respectively.
\rsq
\end{example}

\nt The proof is now complete. \end{proof}


\begin{corollary}\label{cor-3-isocomp} For a one-point union of cycles $G$,  $\chi_{la}(3G)=2$  if and only if
\begin{enumerate}[(a)]
\item $G\cong C((12t)^{[5t+1]}, (12t+2)^{[t]})$, $t\ge 1$, under Case A-1.
\item $G\cong C(2l-2,(4l-4)^{[(2l-2)/3]}, (4l-2)^{[(l-1)/3]})$, $l\equiv 1\pmod{3}\ge 4$, under Case D-1.
\item $G\cong C(4t, (12t)^{[t]}, (12t+2)^{[t]})$, $t\ge 1$, under Case E,
\item $G\cong C(2r, (12r+2)^{[r]})$, $r\ge 2$, under Case F.
\end{enumerate}
\end{corollary}


\section{Graphs in $\mathcal A_2$}\label{sec:2comp}


\nt The ideas used in this section are similar to those in Section~\ref{sec:3comp}. For completeness, all the details are given. Using Lemma~\ref{lem-induced}, we have the following crucial theorem that allows us to obtain all the graphs in $\mathcal A_2$ in this section. Note that we can also get Theorem~\ref{thm-iff-d2} from the proof of Theorem 2.1 in~\cite{LSNZP}.

 \begin{theorem}\label{thm-iff-d2} A graph $G$ with $s\ge 3$ induced paths (or cycles) is in $\widetilde{\mathcal B}_2$ if and only if
\begin{table}[H]
\fontsize{9}{12}\selectfont
\begin{tabular}{l|l|l|l|l}
Case & $s$ & Induced paths & Induced paths & Induced path\\\hline
(1) & $s=4l+2\ge 6$ & $l$ paths of length $4l+2$ & $3l+2$ paths of length $4l$ & not applicable \\\hline
(2a) & $s=3l\ge 6$ & $3l-1$ paths of length $4l-2$ & one path of length $2l-2$ & not applicable \\\hline
(2b)   &  $s=3l-1$ & $l-1$ paths of length $4l-2$  & $2l-1$ paths of length $4l-4$ &  one path of length $2l-2$    \\\hline
(3)$^*$ & $s=2l\ge 4$ & $l-1$ paths of length $4l-2$ & $l$ paths of length $4l-4$ & not applicable \\\hline
(4)$^*$ & not applicable   & $s-2$ paths of length $4s-6$ & one path of length $2s-4$ & not applicable \\\hline
\end{tabular}\\
\normalsize
\nt $^*$ One of these paths is split into two paths of even length at least 4. 
\caption{Exact value of $s$ and the corresponding number of paths and lengths.}\label{table-2-1}
\end{table}

\nt Moreover, the edge labels are unique (up to isomorphism) as follows:
\begin{table}[H]
\fontsize{10}{14}\selectfont
\begin{tabular}{l|l|l}
Case & Path length & Consecutive edge labels      \\\hline
(1) & $4l+2$ & $A_{2l+1}(i;8l+4)\diamond A_{2l+1}(16l^2+10l+1-i; -8l-4)$, $1\le i\le l$ \\
 & $4l$  & $A_{2l}(2l+j; 8l+4)\diamond A_{2l}(16l^2+8l+1-j; -8l-4)$, $1\le j\le 3l+2$ \\\hline
(2a) & $4l-2$ & $A_{2l-1}(i; 6l-1)\diamond A_{2l-1}(12l^2-8l+1-i; -6l+1)$, $1\le i\le 3l-1$  \\
  & $2l-2$ &  $A_{l-1}(6l-1; 6l-1)\diamond A_{l-1}(12l^2-14l+2; -6l+1)$ \\\hline
(2b) & $4l-2$    &  $A_{2l-1}(i; 6l-3) \diamond A_{2l-1}(12l^2-16l+5-i; -6l+3)$, $1\le i\le l-1$  \\
      &  $4l-4$   &   $A_{2l-2}(2l-2+j; 6l-3) \diamond A_{2l-2}(12l^2-18l+7-j; -6l+3)$, $1\le j\le 2l-1$    \\
    &    $2l-2$  &   $A_{l-1}(4l-2; 6l-3) \diamond A_{l-1}(12l^2-20l+7; -6l+3)$                \\\hline
(3) & $4l-2$ & $A_{2l-1}(i; 4l-2) \diamond A_{2l-1}(8l^2-10l+3-i; -4l+2)$, $1\le i\le l$ \\
 & $2l-2$ & $A_{l-1}(2l-2+j; 4l-2)\diamond A_{l-1}(8l^2-12l+5-j; -4l+2)$, $1\le j\le l-1$ \\\hline
(4) & $4s-6$ & $A_{2s-3}(i; 2s-3)\diamond A_{2s-3}(4s^2-12s+9-i; -2s+3)$, $1\le i\le s-3$ \\
 & $2s-4$ & $A_{s-2}(2s-3; 2s-3)\diamond A_{s-2}(4s^2-14s+12; -2s+3)$ \\\hline
\end{tabular}
\caption{Edge labels of each induced path.}\label{table-2-2}
\end{table}
\end{theorem}


\begin{theorem}\label{thm-2comp} A graph $G$ is in $\mathcal A_2$   if and only if $G$ is a graph in one of the following table.

\begin{table}[H]
\fontsize{10}{17}\selectfont
\begin{tabular}{l|l}
Case & $G$ ($1\le \k\le 2$ and $r\ge 2$ unless stated otherwise, $\n_\k, s_\k, t\ge 0$ if applicable)  \\\hline
(1)  & $H_\k\cong C((4l)^{[8l+5\n_\k]}, (4l+2)^{[\n_\k]})$, $l \in \{2\}\cup [4, \infty)$, $0\le \n_\k \le l$, $\n_\k = \r+5s_\k$, $\r\equiv 3l\pmod{5}$, \\
   & $\n_1+\n_2 = l = 2\r+5(s_1+s_2)$.\\\hline
(2a)  &  $C((4l-2)^{[2l]}) + C((4l-2)^{[l-1]}, 2l-2)$, $l\ge 2$. \\\hline
(2b)  &  $C(2l-2, (4l-4)^{[\frac{3l-3-\n_1}{4}]}, (4l-2)^{[\n_1]}) + C((4l-4)^{[\frac{6l-2-\n_2}{4}]}, (4l-2)^{[\n_2]})$, \\
 & $l\ge 3$, $\n_\k\le l-1$, $\n_\k = \r_\k + 4s_\k$, $l-1=\r_1+\r_2+4(s_1+s_2)$.  \\\hline
(3)  &  $C(2r, (4l-4)^{[t]}, (4l-2)^{[2r-3t]}) + C(4l-4-2r, (4l-4)^{[l-1-t]}, (4l-2)^{[3t+l-1-2r]})$, \\
   & $2r-l+1\le 3t\le 2r \le 2l-2$. \\\hline
  &  $C(2r, (4l-4)^{[t]}, (4l-2)^{[2r-3t]}) + C(4l-2-2r, (4l-4)^{[l-t]}, (4l-2)^{[l+3t-2-2r]})$, \\
   &  $2r-l+2 \le 3t \le 2r \le 2l-2$. \\\hline
 (4)  & $C(2r, (4s-6)^{[r]}) + C(4s-6-2r, (4s-6)^{[2s-3-r]},2s-4)$, $r\le s-3$. \\\hline
   & $C(2r, (4s-6)^{[r]})+ C(2s-4-2r, (4s-6)^{[s-2-r]})$, $2\le r\le (s-2)/2$. \\\hline
\end{tabular}
\caption{2-component graphs of one-point unions of cycles}\label{tab:2com}
\end{table}
\end{theorem}

\begin{proof} From Theorem~\ref{thm-iff-d2} (or the proof of Theorem 2.1~\cite{LSNZP}), we only need to obtain all the disconnected one-point unions of cycles solutions under each of the cases as follows. In Case~(1) and Case~(2), we assume $\m_\k, \n_\k, s_\k\ge 0$ for $\k=1,2$. In all other cases, we assume $r\ge 2, t\ge 0$ unless stated otherwise.

\ms\nt {\bf Case (1).}
In this case, $s=4l+2\ge 6$, $y=16l^2+18l+5=(8l+5)(2l+1)$ and $D = [1,8l+4]$. There are $l$ induced path(s) of length $4l+2$, namely, $(i,2l+1-i)_{4l+2}$, $1\le i\le l$; and $3l+2$ induced paths of length $4l$, namely, $(2l+j,8l+5-j)_{4l}$, $1\le j\le 3l+2$. Thus, the sum of two end-edge labels of an induced path of length $4l+2$ is $2l+1$; and that of an induced path of length $4l$ is $5(2l+1)$.

\ms\nt Suppose $X_\k$ contains $\m_\k$ induced paths of length $4l$, and $\n_\k$ induced paths of length $4l+2$, where $0\le \m_\k \le 3l+2$, $0\le \n_\k \le l$. Therefore, $S(X_\k) = 5(2l+1)\m_\k + (2l+1)\n_\k = (2l+1)(5\m_\k+\n_\k)$.  Under the requirement, $5\m_\k + \n_\k = 8l+5$, $\m_1+\m_2 = 3l+2$ and $\n_1+\n_2 = l$. By solving these linear equations, we have
\[\left\{\begin{aligned}
\mu_2 & = 3l+2-\mu_1\\
\nu_1 & = 8l+5-5\mu_1\\
\nu_2 & = 5\m_1-7l-5.\\
\end{aligned}\right.\mbox{ or equivalently }
\left\{\begin{aligned}
\m_1 &= \textstyle\frac{8}{5}l - \frac{1}{5}\n_1+1\\
\m_2 &= \textstyle\frac{7}{5}l + \frac{1}{5}\n_1+1\\  
\n_2 &= l -\n_1
\end{aligned}\right.\]
Since $0\le \n_\k \le l$, $7l+5\le 5\m_\k \le 8l+5$ or equivalent to $l + 1+\frac{2l}{5} \le \m_\k\le l +1+ \frac{3l}{5}$. Thus, we can obtain all the solutions if and only if we obtain all the possible $\m_1$ (or $\n_1$) in the required range.

\ms\nt  Since $5\m_\k = 8l+5-\n_\k$, $8l+5-\n_\k \equiv 0\pmod{5}$, i.e., $\n_\k\equiv3l \pmod{5}$. Thus, for each $\n_\k \equiv 3l \pmod{5}$, a solution exists if and only if $0\le \n_\k\le l$.  By symmetry, we may assume $\n_1\ge \n_2$. Let $\r$ be the least nonnegative residue of $\n_\k \equiv 3l$ modulo 5 for $\k=1,2$, i.e., $\r\in [0,4]$. Thus, $\r=3, 1, 4, 2, 0$ when $l\equiv 1, 2, 3, 4, 0 \pmod{5}$, respectively. Therefore, \[ l = \n_1+\n_2 \ge 2\r = 6, 2, 8, 4, 0 \mbox{ when } l \equiv 1,2,3,4,0\pmod{5}.\] Thus, there is no valid $\n_\k$ when $l = 1, 3 $. For $l = 2,4$, we have $(\n_1,\n_2) = (1,1), (2,2)$, respectively. Note that each of the solutions is unique.

\ms\nt Let $\n_\k = \r + 5s_\k$, where $s_\k\ge 0, \k=1,2$. So, $l = 2\r + 5(s_1 + s_2)$. Suppose $l = 5t + \r_0\ge 5$, where $t\ge 1$ and $0\le \r_0\le 4$. Thus,
\[\begin{array}{rl|*{5}{c|}}
& \r_0= & 0 & 1 & 2 & 3 & 4 \\\hline
\Rightarrow &\r= & 0 & 3 & 1 & 4 & 2\\\hhline{=======|}
\mbox{we may choose } (s_1, s_2) &= & (t,0) & (t-1,0) & (t,0) & (t-1, 0) & (t,0)
\end{array}\]

\ms\nt Thus, there is always a solution for $(\m_\k, \n_\k)$ if and only if $l \in \{2\}\cup [4, \infty)$. If we want to find all non-isomorphic solutions, we need to solve the system completely.  In general, all the solutions are given by 
 \[ H_\k \cong C((4l)^{[8l+5\n_\k]}, (4l+2)^{[\n_\k]}), \] 
where $l \in \{2\}\cup [4, \infty)$, $\n_\k \le l$, $\n_\k = \r+5s_\k$, $\r\equiv 3l\pmod{5}$, $\n_1+\n_2 = l = 2\r+5(s_1+s_2)$. Note that both the components are isomorphic if and only if $\m_1=\m_2=(3l+2)/2$ and $\n_1=\n_2=l/2$, $l\equiv 0 \pmod{2}$, i.e., $G\cong 2C((4l)^{[(3l+2)/2]}, (4l+2)^{[l/2]})$, $l\equiv 0\pmod{2}$ and $l \ge 2$. 

\begin{example} For $l=5\equiv 0\pmod{5}$, $\r=0$ so that $s_1=1,s_2=0$ and $\n_1=5$, $\n_2=0$, $\m_1=8$, $\m_2=9$. The only solution is $C(20^{8]}, 22^{[5]})  + C(20^{[9]})$.



\ms\nt For $l=12\equiv 2\pmod{5}$, $\r=1$ so that $(s_1,s_2)=(2,0)$, $(1,1)$ and $(\n_1,\n_2) = (11,1), (6,6)$, respectively. All the non-isomorphic solutions are  $C(48^{[14]}, 50^{[11]}) + C(48^{[24]}, 50)$ and 2$(48^{[19]}, 50^{[6]})$.   \rsq
\end{example}

\ms\nt {\bf Case (2a).}
In this case, $s=3l\ge 6$, $y=2l(6l-1)$ and $D=[1,6l-1]\cup\{l(6l-1)\}$. There are $3l-1$ induced paths of length $4l-2$, namely $(i, 6l-1-i)_{4l-2}$, $1\le i\le 3l-1$; and an induced path of length $2l-2$, namely $(6l-1, l(6l-1))_{2l-2}$. (Note that in~\cite[P. 380, line 2]{LSZPN} the $l-2$ should be $2l-2$.) Thus, the sum of the two end-edge labels of an induced path of length $4l-2$ is $6l-1$, and that of an induced path of length $2l-2$ is $(l+1)(6l-1)$.

\ms\nt Suppose $X_\k$ contains $\mu_\k$ induced paths of length $4l-2$, and $\nu_\k$ induced paths of length $2l-2$, where $0\le \mu_\k\le 3l-1$ with $\m_1+\m_2=3l-1$ and $\n_1+\n_2=1$. Without loss of generality, assume  $(\n_1, \n_2) = (0,1)$. Therefore, $S(X_1)=(6l-1)\mu_1=2l(6l-1)$ and $S(X_2) = (6l-1)\m_2+(l+1)(6l-1) = (6l-1)(\m_2+l+1)$. Thus, $\m_1=2l$, $\m_2=l-1$. Therefore, the unique solution is $C((4l-2)^{[2l]}) + C((4l-2)^{[l-1]}, 2l-2)$, $l\ge 2$, and both components are not isomorphic. 

\ms\nt{\bf Case (2b)}
In this case, $s=3l-1\ge 5$, $y=(2l-1)(6l-2)$ and $D =[1,6l-3]\cup\{(2l-1)(3l-1)\}$. There are $l-1$ induced path(s) of length $4l-2$, namely, $(i, 2l-1-i)_{4l-2}$, $1\le i\le l-1$ with end-edge labels sum $2l-1$; $2l-1$ induced paths of length $4l-4$, namely, $(2l-2+j, 6l-2-j)_{4l-4}$, $1\le j\le 2l-1$ with end-edge labels sum $4(2l-1)$; and an induced path of length $2l-2$, namely, $(2(2l-1), (2l-1)(3l-1))_{2l-2}$ with end-edge labels sum $(3l+1)(2l-1)$. 

\ms\nt Without loss of generality, suppose $X_1$ contains $\m_1$ induced paths of length $4l-4$, $\n_1$ induced paths of length $4l-2$ and the only path of length $2l-2$, whereas $X_2$ contains $\m_2$ induced paths of length $4l-4$ and $\n_2$ induced paths of length $4l-2$, where $0\le \m_\k\le 2l-1$, $0\le \n_\k\le l-1$, $1\le \k \le 2$. Thus, $S(X_1) = 4(2l-1)\m_1 + (2l-1)\n_1 + (3l+1)(2l-1) = y$ and $S(X_2) = 4(2l-1)\m_2 + (2l-1)\n_2 = y$. Therefore,  $4\m_1+\n_1 + 3l+1 = 6l-2$, i.e., $4\m_1+\n_1=3l-3$ and $4\m_2 + \n_2 = 6l-2$ with $\m_1+\m_2 = 2l-1$ and $\n_1+\n_2 = l-1$. By solving these linear equations, we have
\[\left\{\begin{aligned}
\m_2 & = 2l-1-\m_1\\
\n_1 & = 3l-3-4\m_1\\
\n_2 & =4\m_1-2l+2.\\
\end{aligned}\right.\]
Since $0\le \n_\k \le l-1$, we have $2l-2 \le 4\m_1\le 3l-3$ and $5l-2 \le 4\m_2 \le 6l-3$.

\ms\nt  Since $4\m_1 = 3l-3-\n_1$ and $4\m_2=6l-2-\n_2$, $\n_1\equiv 3l-3\pmod{4}$ and $\n_2\equiv 2l-2\pmod{4}$. Thus, for each valid $\n_1\equiv 3l-3\pmod{4}$ and $\n_2\equiv 2l-2\pmod{4}$, a solution exists if and only if $0\le \n_\k\le l$. Let $\r_1$ (and $\r_2$) be the least nonnegative residue of $3l-3$ (and $2l-2$) modulo 4, i.e.,  $\r\in [0,3]$. Thus,
\[\r_1=
\begin{cases}
0 & \mbox{ if } l\equiv 1\pmod 4,\\
3 & \mbox{ if } l\equiv 2\pmod 4,\\
2 & \mbox{ if } l\equiv 3\pmod 4,\\
1 & \mbox{ if } l\equiv 0\pmod 4;
\end{cases}
\quad \r_2=
\begin{cases}
0 & \mbox{ if } l\equiv 1\pmod 4,\\
2 & \mbox{ if } l\equiv 2\pmod 4,\\
0 & \mbox{ if } l\equiv 3\pmod 4,\\
2 & \mbox{ if } l\equiv 0\pmod 4.
\end{cases}
\]
This implies that \[l = \n_1+\n_2\ge \r_1+\r_2 = \begin{cases}
0 & \mbox{ if } l\equiv 1\pmod{4},\\
5 & \mbox{ if } l\equiv 2\pmod{4},\\
2 & \mbox{ if } l\equiv 3\pmod{4},\\
3 & \mbox{ if } l\equiv 0\pmod{4}.
\end{cases}\]
Therefore, when $l = 2$, there is no valid $\n_\k$. Following are all solutions for $l \in [3,5]$.
\[\begin{array}{c|l}
l & (\n_1, \n_2)\\\hline
3 &  (2,0)\\
4 &  (1,2)\\
5 &  (4,0), (0,4)\\
\end{array}\]

\nt Now, let $\n_\k =\r_\k + 4s_\k$ for some $s_\k\ge 0$, $\k=1,2$. So, $l-1 = \r_1 + \r_2 + 4(s_1+s_2)$. Suppose $l=4t+\r\ge 6$ where $t\ge 1$ and $0\le \r\le 3$. Therefore, we may choose
\[\begin{array}{c|*{4}{c|}}
& \r=0 & \r=1 & \r=2 & \r=3  \\\hline
(s_1, s_2)= & (t,0) & (t,0) & (t-1,0) & (t, 0)
\end{array}\]
Therefore, there is always a solution for $l\ge 3$, and both the components are not isomorphic.   In general, all the solutions are  \[C(2l-2, (4l-4)^{[\frac{3l-3-\n_1}{4}]}, (4l-2)^{[\n_1]}) + C((4l-4)^{[\frac{6l-2-\n_2}{4}]}, (4l-2)^{[\n_2]}),  \] where $l\ge 3$, $\n_\k\le l-1$, $\n_\k = \r_\k + 4s_\k$, $l-1=\r_1+\r_2+4(s_1+s_2)$. 

\begin{example}
For $l=9\equiv 1\pmod{4}$, $\r=0$ so that  $\r_1=\r_2=0$ with corresponding $(\m_1,\m_2)=(5,12), (4, 13), (6,11)$ and $(\n_1,\n_2)=(4,4), (8,0), (0,8)$. Thus, the three solutions are $C(16,32^{[5]},34^{[4]}) + C(32^{[12]},34^{[4]})$, $C(16,32^{[4]},34^{[8]}) + C(32^{[13]})$, and $C(16,32^{[6]}) + C(32^{[11]},34^{[8]})$ respectively.
\end{example}

\ms\nt {\bf Case (3).}
In this case, $s=2l\ge 4$, $y=8l^2-6l+1 = (2l-1)(4l-1)$ and $D=[1,4l-2]\cup\{\g_1,\g_2 \mid \g_1+\g_2=y \mbox{ and } \g_1,\g_2 \mbox{ not adjacent to any end-edge labels}\}$. There are $l-1$ induced paths of length $4l-2$, namely $(i, 2l-1-i)_{4l-2}$, $1\le i\le l-1$, with end-edge labels sum $2l-1$; and $l$ induced paths of length $4l-4$, namely $(2l-2+j, 4l-1-j)$, $1\le j\le l$, with end-edge labels sum $3(2l-1)$. Moreover, one of the paths must be split into two paths of even length at least 4. Thus, $l\ge 3$.

\ms\nt Consider the induced paths of length $4l-4$ with consecutive edge labels $A_{2l-2}(2l-2+j; 4l-2 )\diamond A_{2l-2}(8l^2-12l+5-j; -4l+2)$, $1\le j\le l-1$.   Suppose this path is split into two paths of length $2r$ and $4l-4-2r$ for $2\le r\le 2l-4$, namely, $(2l-2+j, 8l^2-8l+3-j-r(4l-2))_{2r}$ with end-edge labels sum, say $\a_1 = (4l-1-2r)(2l-1)$ and $(2l-2+j+r(4l-2), 4l-1-j)_{4l-4-2r}$ with end-edge labels sum, say, $\a_2 =  (3+2r)(2l-1)$. Note that $\{(4l-1-2r, 3+2r) \mid 2\le r\le l-1\} = \{(3+2r, 4l-1-2r) \mid l-1 \le r\le 2l-4\}$. By symmetry, we need only consider $2\le r\le l-1$. Without loss of generality, we may just take $j=1$ to get the paths $(2l-1, (4l-2-2r)(2l-1))_{2r}$ and $((1+2r)(2l-1),2(2l-1))_{4l-4-2r}$.  Since the induced paths of length $2r$ and $4l-4-2r$ must not belong to the same component of each solution, all the solutions are \[C(2r, (4l-4)^{[t]}, (4l-2)^{[2r-3t]}) + C(4l-4-2r, (4l-4)^{[l-1-t]}, (4l-2)^{[3t+l-1-2r]})\] where $2r-l+1\le 3t\le 2r \le 2l-2$, and both the components are isomorphic if and only if $r=2t=l-1\ge 2$, i.e., $G \cong 2C(2l-2, (4l-4)^{[(l-1)/2]}, (4l-2)^{[(l-1)/2]})$, $l\equiv 1\pmod{2}$ and $l\ge 3$.

\begin{example}  For $l=3$, we get $r=2$, $t=1$. The only solution is $C(4, 8, 10) + C(4, 8, 10)$. For $l=4$, we get $2r-3\le 3t\le 2r \le 6$ so that $r=2,3$ and $(r,t)=(2,1), (3,1), (3,2)$. The solutions are $C(4,12,14) + C(8,12^{[2]},14^{[2]})$ and $C(6,12,14^{[3]})+ C(6,12^{[2]})$  since $(r,t)=(3,1)$ and $(3,2)$ give isomorphic solution. \rsq
\end{example}

\ms\nt Consider the induced paths of length $4l-2$ with consecutive edge labels $A_{2l-1}(i;4l-2)\diamond A_{2l-1}(8l^2-10l+3-i;-4l+2)$, $1\le i\le l$. By the same argument as above, we split the induced path $A_{2l-1}(1,4l-2)\diamond A_{2l-1}(8l^2-10l+2; -4l+2)$ to get $(1, 8l^2-10l+2 - (r-1)(4l-2))_{2r}$ with end-edge labels sum $\a_1 = (4l-1-2r)(2l-1)$, and $(1+r(4l-2), 2l-2)_{4l-2-2r}$ with end-edge labels sum $\a_2 = (1+2r)(2l-1)$ for $2\le r\le l-1$. Since the induced paths $(1, 8l^2-10l+2 - (r-1)(4l-2))_{2r}$ and $(1+r(4l-2), 2l-2)_{4l-2-2r}$ must not belong to the same component of each solution, all the solutions are \[C(2r, (4l-4)^{[t]}, (4l-2)^{[2r-3t]}) + C(4l-2-2r, (4l-4)^{[l-t]}, (4l-2)^{[l+3t-2-2r]})\] where $2r-l+2 \le 3t \le 2r \le 2l-2$, and both components are not isomorphic for all $r,l$.   \rsq 

\begin{example}
\ms\nt For $l=3$, we get $r=2$, $t=1$. the only solution is $C(4,8,10) + C(6,8^{[2]})$.
\end{example}

\ms\nt {\bf Case (4).}
In this case, $s\ge 4$, $y=4s^2-10s+6 = (2s-2)(2s-3)$ and $D=[1,2s-3]\cup\{y/2, \g_1,\g_2 \mid  \g_1+\g_2=y \mbox{ and } \g_1,\g_2 \mbox{ not adjacent to any end-edge labels}\}$. There are $s-2$ induced paths of length $4s-6$, namely $(i,2s-3-i)_{4s-6}$, $1 \le i\le s-2$, with end-edge labels sum $2s-3$; and an induced path of length $2s-4$, namely $(2s-3, (s-1)(2s-3))_{2s-4}$ with end-edge labels sum $s(2s-3)$.

\ms\nt Consider the induced paths of length $4s-6$. By the same argument as in Case (3), we split the induced path $A_{2s-3}(1;2s-3) \diamond A_{2s-3}(4s^2-12s+8; -2s+3)$ to get $(1, 4s^2-10s+5-r(2s-3))_{2r}$ with end-edge labels sum $\a_1 = (2s-2-r)(2s-3)$, and $(1+r(2s-3), 2s-4)_{4s-6-2r}$ with end-edge labels sum $\a_2 = (1+r)(2s-3)$ for $2\le r\le s-2$. Since the induced paths $(1, 4s^2-10s+5-r(2s-3))_{2r}$ and  $(1+r(2s-3), 2s-4)_{4s-6-2r}$ must not belong to the same component of each solution, all the solutions are \[C(2r, (4s-6)^{[r]}) + C(4s-6-2r, (4s-6)^{[2s-3-r]},2s-4)\] where $r\le s-3$. Note that $s\ge 5$ and both the components are not isomorphic for all $r,s$. Observe that  \[C(2r, (4s-6)^{[r-s]}, 2s-4) + C(4s-6-2r, (4s-6)^{[2s-r-3]})\] where $s\le r\le 2s-3$ is not a solution since it contradicts $r\le s-2$. 

\begin{example}  For $s=5$, we get $r=2$. The only solution is $C(4, 14^{[2]}) + C(10, 6)$.  \rsq
\end{example}


\ms\nt Consider the induced paths of length $2s-4$, $s\ge 6$. Suppose we split it to get $(2s-3, (2s-3-r)(2s-3))_{2r}$ with end-edge labels sum $\a_1 = (2s-2-r)(2s-3)$, and $((r+1)(2s-3), (s-1)(2s-3))_{2s-4-2r}$ with end-edge labels sum $\a_2 =(s+r)(2s-3)$ for $2\le r\le (s-2)/2$. Since the induced paths $(2s-3, (2s-3-r)(2s-3))_{2r}$ and $((r+1)(2s-3), (s-1)(2s-3))_{2s-4-2r}$ must not belong to the same component of each solution, all the solutions are \[C(2r, (4s-6)^{[r]})+ C(2s-4-2r, (4s-6)^{[s-2-r]})\] where $2\le r\le (s-2)/2$. Note that $s \ge 6$ and both the components are isomorphic if and only if $r = (s-2)/2$, i.e., $G\cong 2C(s-2, (4s-6)^{[(s-2)/2]})$, $s\equiv 0\pmod{2}$.

\begin{example} For $s=6$, we get $r=2$. The only solution is $2C(4, 18^{[2]})$.  \rsq
\end{example}  \end{proof}

\begin{corollary}\label{cor-2-isocomp} For a one-point union of cycles $G$,  $\chi_{la}(2G)=2$  if and only if
\begin{enumerate}[(i)]
\item $G\cong C((4l)^{[(3l+2)/2]}, (4l+2)^{[l/2]})$, $l\equiv 0\pmod{2}$ and $l\ge 2$,
\item $G\cong C(2l-2, (4l-4)^{[(l-1)/2]}, (4l-2)^{[(l-1)/2]})$, $l\equiv 1\pmod{2}$ and $l\ge 3$,
\item $G\cong C(l-2, (4l-6)^{[(l-2)/2]})$, $l\equiv 0\pmod{2}$ and $l \ge 6$. 
\end{enumerate}
\end{corollary}

\section{Conclusion and Open Problems}
The complete characterizations established in this paper demonstrate the efficacy of mapping structural graph labeling constraints directly onto bounded algebraic condition sets. Rather than treating disconnected topologies as isolated combinatorial puzzles, this unified partition framework transforms complex boundary conditions into systemic invariant relationships. This methodology provides a repeatable blueprint for tackling classification bottlenecks in graph labeling invariants across general families of disjoint graph unions.


\ms\nt In~\cite[Theorem 2.4]{LSN-IJMSI}, \cite{Lau+Shiu+chila2d=3} and this paper, we prove that using a case-by-case algebraic analysis, one can obtain all non-isomorphic families of disconnected one-point unions of cycles  in $A_d, d\ge 1$.  The cases where $d\in\{1,2,3\}$ are completely solved. Theorems~\ref{thm-iff-d3} and~\ref{thm-iff-d2} give partial solutions to Problem~\ref{prob-Bd}.

\ms\nt  For $d=2,3$, if $G$ is a graph in $\widetilde{\mathcal B}_d$ not satisfying one of the conditions in Theorems~\ref{thm-iff-d3} and~\ref{thm-iff-d2} respectively, then $\chi_{la}(G)\ge 3$. Thus, for $d=2,3$, Problem~\ref{prob-Bd} is completely solved if the following problem is settled.

\begin{problem} Completely characterize all graphs $G$ in $\widetilde{\mathcal B}_d\setminus \mathcal A_d$ for $d=2,3$.  \end{problem}

\nt The following problem arises naturally.

\begin{problem} For $d\ge 4$, characterize all graphs in $\mathcal A_d$ without a case-by-case analysis. \end{problem}

\nt In~\cite[Theorem 8]{LNS-GC}, the authors proved the existence of a connected graph $G$ of order $n$ with pendant vertices such that $\chi_{la}(G) = n-k$ if and only if  $n\ge k+3\ge 3$.  Since $\chi_{la}(3C(2r, (12r+2)^{[r]})) = \chi(3C(2r, (12r+2)^{[r]})) = 2$ for $r\ge 2$, the following conjecture arises naturally.

\begin{conjecture}\label{conj-discon}
For connected graphs $G$ of order $n\ge 3$ without pendant vertices, there exist graphs $mG$ such that $\chi_{la}(mG) = mn-k$ for each $m\ge 2$ and $0\le k\le mn-2$. Moreover, characterize all such graphs $G$.
\end{conjecture}

\nt By  Corollaries~\ref{cor-3-isocomp} and~\ref{cor-2-isocomp}, we know Conjecture~\ref{conj-discon} holds for $m=2,3$ and $k=mn-2$. 

\ms\nt Note that $\mathcal A_1$ has 2, $\mathcal A_2$ has 7,  while $\mathcal A_3$ has 28 distinct families of non-isomorphic solutions. Moreover, there are 3 and 4 distinct families with all isomorphic components in $A_2$ and $A_3$, respectively. We conclude with the following conjecture.

\begin{conjecture}  As $d\to \infty$, $\mathcal A_d$ has infinitely many distinct families of solutions (with or without all isomorphic components).  \end{conjecture}

\ms\nt{\bf\fontsize{15pt}{30pt}\selectfont Acknowledgements} 

\ms\nt The first author dedicates this paper to Profs. J.D.H. Smith and  S.Y. Song of Iowa State University for introducing him to the beauty of abstract algebra, graph theory, and combinatorics.

\end{document}